\pdfoutput=1
\RequirePackage{ifpdf}
\ifpdf 
\documentclass[pdftex]{sigma}
\else
\documentclass{sigma}
\fi

\usepackage{euscript}
\usepackage[affil-it]{authblk}
\usepackage[all]{xy}

\newcommand{\RN}{\mathbb{R}} 
\newcommand{\CN}{\mathbb{C}} 
\newcommand{\ZN}{\mathbb{Z}} 

\newcommand{\eps}{\ensuremath\varepsilon}

\newcommand{\prodl}{\prod\limits}
\newcommand{\ovl}{\overline}

\newcommand{\gr}{\mathop{\mathrm{gr}}\nolimits}

\newcommand{\SL}{\mathop{\mathrm{SL}}\nolimits}

\renewcommand{\Im}{\mathop{\mathrm{Im}}\nolimits}

\newcommand{\ad}{\ensuremath\operatorname{ad}}

\newcommand{\Span}{\operatorname{Span}}

\newcommand{\mf}{\mathfrak}

\newcommand{\mc}[1]{\mathcal{#1}}

\renewcommand{\Re}{\operatorname{Re}}

\renewcommand{\th}{\vartheta}
\newcommand{\rmi}{\mathrm{i}}
\newcommand{\abs}[1]{\lvert #1\rvert}

\numberwithin{equation}{section}

\newtheorem{thr}{Theorem}[section]
\newtheorem*{thr*}{Theorem}
\newtheorem*{thr1*}{Theorem~\ref{ThrConeAsASetAnnulus}}
\newtheorem*{thr2*}{Theorem~\ref{ThrConeInGeneralCase}}
\newtheorem*{thr3*}{Theorem~\ref{ThrWeilGenericTraceIsNondeg}}
\newtheorem{lem}[thr]{Lemma}
\newtheorem*{lem*}{Lemma}
\newtheorem{cor}[thr]{Corollary}
\newtheorem{prop}[thr]{Proposition}

\newtheorem*{stat*}{Statement}

\theoremstyle{definition}
\newtheorem{defn}[thr]{Definition}
\newtheorem{Example}[thr]{Example}
\newtheorem{rem}[thr]{Remark}
\newtheorem*{rem*}{Remark}

\begin{document}

\renewcommand{\thefootnote}{}

\newcommand{\arXivNumber}{2105.12652}

\renewcommand{\PaperNumber}{009}

\FirstPageHeading

\ArticleName{Twisted Traces and Positive Forms\\ on Generalized $\boldsymbol{q}$-Weyl Algebras}
\ShortArticleName{Twisted Traces and Positive Forms on Generalized $q$-Weyl Algebras}

\Author{Daniil KLYUEV}
\AuthorNameForHeading{D.~Klyuev}
\Address{Department of Mathematics, Massachusetts Institute of Technology, USA}
\Email{\href{mailto:klyuev@mit.edu}{klyuev@mit.edu}}

\ArticleDates{Received May 27, 2021, in final form January 17, 2022; Published online January 30, 2022}

\Abstract{Let $\mc{A}$ be a generalized $q$-Weyl algebra, it is generated by $u$, $v$, $Z$, $Z^{-1}$ with relations $ZuZ^{-1}=q^2u$, $ZvZ^{-1}=q^{-2}v$, $uv=P\big(q^{-1}Z\big)$, $vu=P(qZ)$, where $P$ is a~Laurent polynomial. A Hermitian form $(\cdot,\cdot)$ on $\mc{A}$ is called invariant if $(Za,b)=\big(a,bZ^{-1}\big)$, $(ua,b)=(a,sbv)$, $(va,b)=\big(a,s^{-1}bu\big)$ for some $s\in \CN$ with $|s|=1$ and all $a,b\in \mc{A}$. In this paper we classify positive definite invariant Hermitian forms on generalized $q$-Weyl algebras.}


\Classification{17B37; 53D55; 81R10}

\section{Introduction}
Let $\mc{A}$ be an algebra over $\CN$, $g$ be an automorphism of $\mc{A}$. We say that a linear map $T\colon \mc{A}\to\CN$ is a $g$-twisted trace if $T(ab)=T(bg(a))$ for all $a,b\in\mc{A}$. Let $\rho$ be an antilinear automorphism of $\mc{A}$. We say that a linear map $T\colon\mc{A}\to\CN$ is a positive trace if $(a,b)_T=T(a\rho(b))$ is a positive definite Hermitian form on $\mc{A}$. The hermitian condition implies that $T$ is indeed a $\rho^2$-twisted trace.

The classification of positive traces on quantizations is an important and interesting question. For example, let $\mf{g}$ be a complex simple Lie algebra, $G$ be the corresponding simply connected group. For central quotients of $U(\mf{g})$ this question is equivalent to the classification of spherical unitarizable $G$-modules of principle series.

Let $C_n$ be a cyclic subgroup of $\SL(2,\CN)$ that consist of matrices $\left(\begin{smallmatrix}
\eps & 0\\0 &\eps^{-1}
\end{smallmatrix}\right)$ such that $\eps^n=1$. It~naturally acts on a polynomial algebra $\CN[x,y]$. We can consider the subalgebra of invariant polynomials $\CN[x,y]^{C_n}$, the associated algebraic variety is called a Kleinian singularity of type~$A$. Positive traces on quantizations of Kleinian singularities of type~$A$ were studied in~\cite{BPR,DFPY,DPY,EKRS,K}. The paper~\cite{EKRS} gives a complete classification of positive traces.

We will consider the natural $q$-analogue of that situation and study positive traces on $q$-deformations of Kleinian singularities of type~$A$. The algebra $\CN[x,y]^{C_n}$ admits only one Poisson bracket up to a scaling, so we use terms ``deformation'' and ``quantization'' interchangeably.

The $q$-analogue of a deformation of $\CN[x,y]^{C_n}$ is the following algebra $\mc{A}$. It~is generated by~$u$,~$v$, $Z$, $Z^{-1}$ with relations $ZuZ^{-1}=q^2u$, $ZvZ^{-1}=q^{-2}v$, $uv=P\big(q^{-1}Z\big)$, $vu=P(qZ)$, where $P$ is a Laurent polynomial. These algebras are also called generalized $q$-Weyl algebras~\cite{Bav}: when $P(z)=1$, $\mc{A}$ is a $q$-Weyl algebra.

When $P(z)=az^{-1}+b+cz$, the algebra $\mc{A}$ is a central reduction of $U_q(\mf{sl}_2)$. We expect that there exists a connection between our results and the classification of unitarizable Harish-Chandra $U_q(\mf{sl}_2)$-bimodules~\cite{Pusz}.

Positive traces for deformations of $\CN[x,y]^{C_n}$ appear in 3-dimensional superconformal field theories~\cite{BPR}. Positive traces for $q$-deformations of $\CN[x,y]^{C_n}$ appear in the study of the Coulomb branch of 4-dimensional superconformal field theories~\cite{DG}.

The structure of the paper is as follows.
In Section~\ref{SecAnalytical} we prove analytic formulas for twisted traces on $q$-deformations similar to those in \cite[Section~3]{EKRS}. The main difference is that now $T(R)=\int_0^1 R\big({\rm e}^{2\pi {\rm i} x}\big)w(x)\,{\rm d}x$ plus some additional terms, where $w$ is quasiperiodic in two directions: $w(x+1)=w(x)$, $w(x+\tau)=tw(x)$. Here~$q={\rm e}^{\pi\rmi \tau}$.
After that we classify positive traces. The conjugation $\rho$ that we use is defined only for real~$q$. Since $q\neq 1$ and $q$, $-q$, $q^{-1}$ all give isomorphic algebras we assume that $0<q<1$ in Sections~\ref{SecAnnulus} and~\ref{SecNotAnnulus}.
Positive traces naturally form a convex cone, so we will often use the sentence ``the cone of positive traces'' instead of ``the set of positive traces''. When we say that it is isomorphic to another cone we mean that there is a bijection that preserves addition of elements and multiplication by a positive real number.

In Section~\ref{SecAnnulus} we classify positive traces in the case when all roots $\alpha$ of $P$ satisfy $q<|\alpha|<q^{-1}$. The answer is a cone of quasiperiodic functions that are positive on $\RN$ and satisfy a certain condition on $\RN+\frac{\tau}{2}$. More precisely, we have the following theorem:

\begin{thr1*}
The cone $C$ of positive $g_t$-twisted traces has dimension $n$. It~is isomorphic to the cone of entire functions $f$ such that
\[
f(x+1)=f(x), \qquad f(x+\tau)={\rm e}^{-\pi {\rm i} (2n x+n\tau-2\sum\beta_j-2c)}f(x)
\]
and
\[
f(x), \qquad {\rm e}^{\pi {\rm i}(-nx+\sum\beta_j+m_0+c)}f\left(x+\frac{\tau}{2}\right)
\]
are nonnegative on $\RN$. At the level of sets $C$ consists of functions
\[
\lambda\frac{\th(z-a_1)\cdots\th(z-a_n)}{\th(z-\beta_1)\cdots\th(z-\beta_n)},
\]
where $\lambda>0$, $\sum a_i-\sum\beta_i-c-m_0\in 2\ZN$ and all $a_i$ are divided into pairs $(a_i,a_j)$ with $a_i=\ovl{a_j}$. In particular, if $Q$ is another Laurent polynomial with the same number of nonzero roots $n$ counted with multiplicities, $s$ is a nonzero complex number with $\abs{s}=1$, cones $C_{P,t}$ and $C_{Q,s}$ are isomorphic.
\end{thr1*}

Here $c$ is a number such that $t={\rm e}^{2\pi \mathrm{i} c}$, $m_0$ is either $0$ or $1$ depending on the behavior of~$P$ on~$S^1$. The main difference with~\cite{EKRS} is that the dimension of the cone always equals to the number of roots of $P$ counted with multiplicities.

In Section~\ref{SecNotAnnulus} we classify positive traces in the general case. The reasoning in this section is very similar to the reasoning in \cite[Sections~4.3--4.4]{EKRS} with some simplifications because we integrate over the compact set $S^1$ instead of $\rmi\RN$. We get the following answer.

\begin{thr2*}
Let $\mc{A}$ be a $q$-deformation with parameter $P$. Let $t\in S^1$, $\rho_t$ be the corresponding conjugation. Let~$l$ be the number of roots $\alpha$ of $P$ such that $q<|\alpha|<q^{-1}$, $r$ be the number of distinct roots $\alpha$ with $|\alpha|=q$. Then the cone $\mc{C}_+$ of positive definite $\rho_t$-equivariant traces is isomorphic to $\mc{C}_l\times \RN_{\geq 0}^r$, where $\mc{C}_l$ consists of nonzero entire functions $f$ such that
\begin{enumerate}\itemsep=0pt
\item[$(1)$] $f(x+1)=f(x)$, $f(x+\tau)={\rm e}^{-\pi {\rm i} n (\tau+2x)} f(x)$,
\item[$(2)$] $f(x)$ and ${\rm e}^{\pi {\rm i} n x}f\left(x+\frac{\tau}{2}\right)$ are nonnegative on $\RN$.
\end{enumerate}
\end{thr2*}

In Section~\ref{SecStarProducts} we assume that $P$ belongs to $\CN[z]$. In this case we can consider the subalgebra $\mc{A}_+$ generated by $u$, $v$, $Z$. Algebra $\mc{A}_+$ is filtered and $A_+=\gr\mc{A}_+$ depends only on the degree $m$ of~$P$. Trace $T$ on a filtered algebra $\mc{B}$ is called nondegenerate if bilinear form $(\cdot,\cdot)\colon \mc{B}_{\leq k}\times \mc{B}_{\leq k}\to \CN$ defined by $(a,b)=T(ab)$ is nondegenerate for all positive integers $k$. Motivation for considering nondegenerate traces comes from paper~\cite{ES} that connects nondegenerate traces on $\mc{A}_+$ and short star-products on $A_+$. We obtain the following result:

\begin{thr3*}
Let $\abs{q}\neq 1$. Then we can find a countable subset $Z$ of $\CN$ containing $1, q^{-2},\allowbreak q^{-4},\dots$ with the following property. For any $t\in\CN\setminus Z$ there exists a countable union of algebraic hypersurfaces $X$ in $\CN^{2n+1}$ such that for any $(c_0,\dots,c_n,t_0,\dots,t_{n-1})\in \CN^{2n+1}\setminus X$ the $g_t$-twisted trace $T$ given by $T(Z^i)=t_i$ on the algebra $\mc{A}_+$ with parameter $P(x)=\sum_{i=0}^n c_i x^i$ is nondegenerate.
\end{thr3*}

In Section~\ref{SecSl2} we study in more details the central reductions of $U_q(\mf{sl}_2)$.

\section{An analytic construction of traces}\label{SecAnalytical}
\subsection{Preliminaries}
Let $q$ be a complex number, $P$ be a Laurent polynomial. Then $\mc{A}$ is an algebra generated by~$u$,~$v$,~$Z$ with relations $ZuZ^{-1}=q^2u$, $ZvZ^{-1}=q^{-2}v$, $uv=P\big(q^{-1}Z\big)$, $vu=P(qZ)$. We assume that $\abs{q}\neq 1$. Note that $q$ and $q^{-1}$ give isomorphic algebras: the isomorphism interchanges~$u$ and~$v$ and sends $Z$ to itself. So we will assume that $\abs{q}<1$.
Sometimes we will write ``polynomial'' instead of ``Laurent polynomial''.

Let $n$ be the number of nonzero roots of $P$ counted with multiplicities. In other words, for $P=ax^k+\cdots+bx^l$ with $a$, $b$ nonzero and $k\geq l$ we have $n=k-l$. We will assume that $n>0$.

We have
\[
\mc{A}=\oplus_{i\in \ZN}\mc{A}_i,
\]
$ZaZ^{-1}=q^{2i}a$ for $a\in \mc{A}_i$. The linear subspace $\mc{A}_i$ equals to $u^i\CN\big[Z,Z^{-1}\big]$ for $i\geq 0$ and $v^{-i}\CN\big[Z,Z^{-1}\big]$ for $i\leq 0$.

Let $t$ be a nonzero complex number. There exists an automorphism $g_t$ of $\mc{A}$ such that $g_t(u)=tu$, $g_t(v)=t^{-1}v$, $g_t(Z)=Z$. A linear map $T\colon\mc{A}\to \CN$ is called a $g_t$-twisted trace if~$T(ab)=T(bg_t(a))$ for all $a,b\in \mc{A}$.
\begin{prop}
\label{PropTwistedTraceIsZeroOnPolnomials}
$T$ is a $g_t$-twisted trace if and only if $T$ is supported on $\mc{A}_0$ and
\[
T\big(P\big(q^{-1}Z\big)R\big(q^{-1}Z\big)-tP(qZ)R(qZ)\big)=0
\]
for all $R\in\CN\big[z,z^{-1}\big]$. The space of $g_t$-twisted traces has dimension $n$.
\end{prop}
\begin{proof}
It is enough to check that $T(ab)=T(bg_t(a))$ when $a=u,v,Z$.
The condition $T(Zb)=T(bZ)$ is equivalent to $T$ being supported on $\mc{A}_0$. So we have to check that $T(ub)=tT(bu)$ for $b=vR\big(q^{-1}Z\big)$. For this $b$ we have
\begin{gather*}
ub=uv R\big(q^{-1}Z\big)=P\big(q^{-1}Z\big)R\big(q^{-1}Z\big),
\\
bu=vR\big(q^{-1}Z\big)u=vuR(qZ)=P(qZ)R(qZ).
\end{gather*}
This gives
\[
\big(P\big(q^{-1}Z\big)R\big(q^{-1}Z\big)-tP(qZ)R(qZ)\big)=0.
\]
We also have to check that $T(vb)=t^{-1}T(bv)$ for $b=uR(qZ)$. For this $b$ we have
\begin{gather*}
vb=vuR(qZ)=P(qZ)R(qZ),
\\
bv=uR(qZ)v=uvR\big(q^{-1}Z\big)=P\big(q^{-1}ZR(qZ)\big).
\end{gather*}
This gives
\[
T\big(t^{-1}P\big(q^{-1}Z\big)R\big(q^{-1}Z\big)-P(qZ)R(qZ)\big)=0.
\]
Hence we can view $g_t$-twisted traces as functions
\[
T\colon\ \CN\big[z,z^{-1}\big]\to \CN
\]
satisfying
\[
T\big(P\big(q^{-1}z\big)R\big(q^{-1}z\big)-tP(qz)R(qz)\big)=0.
\]

We turn to the statement about dimension. Consider the map $\phi\colon \CN\big[z,z^{-1}\big]\to\CN\big[z,z^{-1}\big]$ given by $\phi(S(z))=S\big(q^{-1}z\big)-tS(qz)$. We have $\phi\big(z^k\big)=q^{-k}\big(1-tq^{2k}\big)z^k$.

In the case when $t$ is not an integer power of $q^2$ the map $\phi$ is a linear isomorphism. Hence the codimension of $\phi\big(P(z)\CN\big[z,z^{-1}\big]\big)$ in $\CN\big[z,z^{-1}\big]$ equals to the codimension of $P(z)\CN\big[z,z^{-1}\big]$ in~$\CN\big[z,z^{-1}\big]$ which is equal to $n$.

In the case when $t=q^{-2k}$ map $\phi$ has a kernel $z^{k}$, the image of $\phi$ does not contain $z^k$ and has codimension $1$ in $\CN\big[z,z^{-1}\big]$. Since $P$ has nonzero roots $z^k$ does not belong to $P(z)\CN\big[z,z^{-1}\big]$ and the set $z^k+P(z)\CN\big[z,z^{-1}\big]$ has codimension $n-1$ in $\CN\big[z,z^{-1}\big]$. It~follows that $\phi\big(P(z)\CN\big[z,z^{-1}\big]\big)$ has codimension $n-1$ in the image of $\phi$. We deduce that $\phi\big(P(z)\CN\big[z,z^{-1}\big]\big)$ has codimension $n$ in $\CN\big[z,z^{-1}\big]$.

In both cases we get that $\phi(z)\CN\big[z,z^{-1}\big]$ has codimension $n$ in $\CN\big[z,z^{-1}\big]$. The only condition on~$T$ is that $T$ equals to zero on $\phi(P(z)\CN\big[z,z^{-1}\big]$, hence the space of such maps has dimen\-sion~$n$.
\end{proof}

We also note that the set $\big\{P\big(q^{-1}z\big)R\big(q^{-1}z\big)-tP(qz)R(qz)\mid R(z)\in\CN\big[z,z^{-1}\big]\big\}$ does not change when we multiply $P$ by a nonzero complex number or an integer power of $z$, hence the set of $g_t$-twisted traces will be the same

\subsection[The case when all roots of P satisfy |q|<|z|<|q|\textasciicircum{}\{-1\}]
{The case when all roots of $\boldsymbol P$ satisfy $\boldsymbol{\abs{q}<\abs{z}<\abs{q}^{-1}}$}

Choose $\tau\in\CN$ such that $q=\exp(\pi {\rm i} \tau)$. Let
\[
\th(x)=\th(x;\tau)=\sum_{n\in\ZN} q^{n^2}{\rm e}^{2\pi {\rm i} x n}
\]
be the Jacobi theta function. Let $U=\big\{\abs{q}<\abs{z}<\abs{q}^{-1}\big\}$, $V=\frac{1}{2\pi {\rm i}}\ln U$, $D$ be a fundamental region of~$\th(x)$, $D_0$ its interior, $\ovl{D}$ its closure. We choose $D$ so that $\ovl{D}$ is the parallelogram on vertices $-\frac{\tau}2$, $\frac{\tau}2$, $\frac{\tau}2+1$, $1-\frac{\tau}2$.

\begin{thr}\label{ThrTwistedTraceAndQuasiPeriodic}
Suppose that $w$ is a function such that
\begin{gather*}
w(x)=w(x+1),
\\
w(x+\tau)=tw(x),
\end{gather*}
and $w\big(x+\frac{\tau}{2}\big)P\big({\rm e}^{2\pi {\rm i}x}\big)$ is holomorphic when $x\in \ovl{V}$. Then
\[
T(R(z))=\int_0^1 w(x)R\big({\rm e}^{2\pi {\rm i}x}\big)\,{\rm d}x
\]
is a $g_t$-twisted trace. Moreover, if all roots of $P(x)$ belong to $U$ then any twisted trace can be obtained in this way.
\end{thr}

\medskip\noindent
{\bf Proof.}
We have
\begin{gather*}
T\big(P\big(q^{-1}z\big)R\big(q^{-1}z\big)-tP(qz)R(qz)\big)
\\ \qquad
{}=\int_0^1 w(x)P\big({\rm e}^{2\pi {\rm i} \left(x-\frac{\tau}{2}\right)}\big)R\big({\rm e}^{2\pi {\rm i} \left(x-\frac{\tau}{2}\right)}\big)\,{\rm d}x-t\int_0^{1} w(x)P\big({\rm e}^{2\pi {\rm i} \left(x+\frac{\tau}{2}\right)}\big)R\big({\rm e}^{2\pi {\rm i} \left(x+\frac{\tau}{2}\right)}\big)\,{\rm d}x
\\ \qquad
{}= \int_{-\frac{\tau}{2}}^{1-\frac{\tau}{2}} w\left(x+\frac{\tau}{2}\right)P\big({\rm e}^{2\pi {\rm i}x}\big)R\big({\rm e}^{2\pi {\rm i}x}\big)\,{\rm d}x-t\int_{\frac{\tau}{2}}^{1+\frac{\tau}{2}} w\bigg(x-\frac{\tau}{2}\bigg)P\big({\rm e}^{2\pi {\rm i}x}\big)R\big({\rm e}^{2\pi {\rm i}x}\big)\,{\rm d}x.
\end{gather*}

Recall that $\partial D$ is the parallelogram with vertices $-\frac{\tau}{2}$, $\frac{\tau}{2}$, $1+\frac{\tau}{2}$, $1-\frac{\tau}{2}$. We can continue as follows:
\begin{gather*}
\int_{-\frac{\tau}{2}}^{1-\frac{\tau}{2}} w\left(x+\frac{\tau}{2}\right)
P\big({\rm e}^{2\pi {\rm i}x}\big)R\big({\rm e}^{2\pi {\rm i}x}\big)\,{\rm d}x -t\int_{\frac{\tau}{2}}^{1+\frac{\tau}{2}} w\bigg(x-\frac{\tau}{2}\bigg)
P\big({\rm e}^{2\pi {\rm i}x}\big)R\big({\rm e}^{2\pi {\rm i}x}\big)\,{\rm d}x
\\ \qquad
{}=
\int_{-\frac{\tau}{2}}^{1-\frac{\tau}{2}} w\left(x+\frac{\tau}{2}\right)
P\big({\rm e}^{2\pi {\rm i}x}\big)R\big({\rm e}^{2\pi {\rm i}x}\big)\,{\rm d}x -\int_{\frac{\tau}{2}}^{1+\frac{\tau}{2}} w\left(x+\frac{\tau}{2}\right)
P\big({\rm e}^{2\pi {\rm i}x}\big)R\big({\rm e}^{2\pi {\rm i}x}\big)\,{\rm d}x
\\ \qquad
{}=\int_{\partial D}w\left(x+\frac{\tau}{2}\right)P\big({\rm e}^{2\pi {\rm i} x}\big)
R\big({\rm e}^{2\pi {\rm i} x}\big)\,{\rm d}x
-\int_{\frac{\tau}{2}}^{-\frac{\tau}{2}} w\left(x+\frac{\tau}{2}\right)
P\big({\rm e}^{2\pi {\rm i}x}\big)R\big({\rm e}^{2\pi {\rm i}x}\big)\,{\rm d}x
\\ \qquad\hphantom{=}
{}-\int_{1-\frac{\tau}{2}}^{1+\frac{\tau}{2}}w\left(x+\frac{\tau}{2}\right)
P\big({\rm e}^{2\pi {\rm i}x}\big)R\big({\rm e}^{2\pi {\rm i}x}\big)\,{\rm d}x
=\int_{\partial D}\!w\left(x+\frac{\tau}{2}\right)P\big({\rm e}^{2\pi {\rm i} x}\big)
R\big({\rm e}^{2\pi {\rm i} x}\big)\,{\rm d}x=0.
\end{gather*}
We used that $R\big({\rm e}^{2\pi {\rm i}x}\big)P\big({\rm e}^{2\pi {\rm i}x}\big)w\big(x+\frac{\tau}{2}\big)$ is $1$-periodic and $P\big({\rm e}^{2\pi {\rm i}x}\big)w\big(x+\frac{\tau}{2}\big)$ is holomorphic on~$\ovl{D}$.

Suppose that all roots of $P(x)$ belong to $U$. It~remains to prove that any trace $T$ can be obtained from a function $w$ as above. Since different $w$ give different traces it is enough to prove that the space of functions $w$ satisfying the conditions of the theorem has dimension $n$.
All roots of $P\big({\rm e}^{2\pi {\rm i}x}\big)$ belong to the strip $V$. Therefore there are $n$ roots of $P({\rm e}^{{\rm i}x})$ in $D_0$. Denote these roots by $\alpha_1,\dots,\alpha_n$. Let $\beta_i=\alpha_i+\frac{1}{2}$.
Recall that $\th(z)$ has zeroes at points $\frac12+\frac{\tau}{2}+\ZN+\ZN\tau$ and $\th(z+\tau)={\rm e}^{-\pi {\rm i} (\tau+2z)}\th(z)$. For the proof of this fact see, for example, \cite[Lemma 4.1]{Mumford}.

Let $t={\rm e}^{2\pi {\rm i} c}$. We get the result from the following lemma.

\begin{lem}
\label{LemHowLLooks}
Let $S$ be the set of functions $w$ that satisfy
\[
w=\lambda\frac{\th(z-a_1)\cdots \th(z-a_n)}{\th(z-\beta_1)\cdots\th(z-\beta_n)},
\]
for some $\lambda\in \RN$, $a_1,\dots,a_n\in\CN$ such that $\sum a_i-\sum \beta_i\in c+\ZN$. Then any $w\in S$ satisfies con\-ditions of Theorem~$\ref{ThrTwistedTraceAndQuasiPeriodic}$ and $S$ is a linear space of dimension $n$.
\end{lem}

 \begin{proof}
 Let
 \[
 w(z)=\lambda\frac{\th(z-a_1)\cdots \th(z-a_n)}{\th(z-\beta_1)\cdots\th(z-\beta_n)}
 \]
 be an element of $S$. This function satisfies $w(x+1)=w(x)$ and
 \begin{align*}
 w(x+\tau)&=\prod_{k=1}^n {\rm e}^{-\pi {\rm i} (\tau+2z-2a_k)}{\rm e}^{\pi {\rm i} (\tau+2z-2\beta_k)}w(x)=
 {\rm e}^{2\pi {\rm i} (\sum a_k-\sum\beta_k)}w(x)
 \\
 &={\rm e}^{2\pi {\rm i} c}w(x)= tw(x).
 \end{align*}
 Since $\th$ has zeros at points $\frac{1}{2}+\frac{\tau}{2}+\ZN+\ZN\tau$, $w(z+\frac{\tau}{2})$ has poles at points $\alpha_i+\ZN+\ZN\tau$. It~follows that $P\big({\rm e}^{2\pi {\rm i} x}\big)w\big(x+\frac{\tau}{2}\big)$ is holomorphic on $V$.

Multiplication by $\frac{\th(z-\beta_1)}{\th(z-\beta_1-c)}$ gives a bijection between $S$ and the space of elliptic functions with poles in $\alpha_1+c,\alpha_2,\dots,\alpha_n$. It~follows that $S$ is a linear space of dimension $n$.
\end{proof}

\subsection{General case}
\label{SubSecTracesInGeneral}
We start with the case when all roots $\alpha$ of $P$ satisfy $\abs{q}\leq |\alpha|\leq \frac{1}{\abs{q}}$. We will mainly use analytic formulas for traces in the case when $P(z)=\ovl{P}\big(z^{-1}\big)$. So we assume that roots of $P$ on $qS^1\cup q^{-1}S^1$ are symmetric with respect to the map $z\mapsto \ovl{z^{-1}}$. In this case we can find polynomials $P_*$, $Q$ such that $P(x)=P_*(x)Q(qx)\ovl{Q}\big(qx^{-1}\big)$, all nonzero roots of $P_*$ belong to $U$ and all roots of $Q$ belong to $S^1$.

Let us choose a linear subspace $S$ of $\CN\big[z,z^{-1}\big]$ such that the composition $S\subset \CN\big[z,z^{-1}\big]\to \CN\big[z,z^{-1}\big]/(Q)$ is an isomorphism. For any $R\in\CN\big[x,x^{-1}\big]$ we define $R_1\in\CN\big[z,z^{-1}\big]$, $R_0\in S$ by $R(x)=R_1Q(x)+R_0(x)$.
\begin{prop}\label{PropTracesClosedAnnulus}
 Any $g_t$-twisted trace $T$ on $\mc{A}$ can be written as
\[
T(R)=\int_0^1 R_1\big({\rm e}^{2\pi {\rm i} x}\big)Q\big({\rm e}^{2\pi {\rm i} x}\big)w(x)\,{\rm d}x+l(R_0),
\]
where $R_0$, $R_1$ are defined above, $w$ is a meromorphic function such that $w(x+1)=w(x)$, $w(x+\tau)=tw(x)$, $Q\big({\rm e}^{2\pi {\rm i} x}\big)w(x)$ has no poles on $[0,1]$, $P\big({\rm e}^{2\pi {\rm i} x}\big)w\big(x+\frac{\tau}{2}\big)$ has no poles $V$, and~$l$ is an arbitrary linear functional on~$S$.
\end{prop}

\begin{proof}
For $z\in S^1$ we have $\ovl{Q(z)}=\ovl{Q}\big(z^{-1}\big)$. Since all roots of $Q$ belong to $S^1$ it follows that $\ovl{Q}\big(z^{-1}\big)$ has the same roots as~$Q(z)$ with the same multiplicities. Therefore $Q(z)$ divides $\ovl{Q}\big(z^{-1}\big)$.

First we prove that any pair of $w$ and~$l$ gives a trace. We have to show that
\[
T\big(R\big(q^{-1}z\big)P\big(q^{-1}z\big)-tR(qz)P(qz)\big)=0
\]
for any $R\in\CN\big[z,z^{-1}\big]$. We note that $Q(z)$ divides
\[
P\big(q^{-1}z\big)=P_*\big(q^{-1}z\big)Q(z)\ovl{Q}\big(q^2z\big)
\]
and
\[
P(qz)=P_*(qz)Q\big(q^2z\big)\ovl{Q}\big(z^{-1}\big).
\]
It follows that
\begin{gather*}T\big(R\big(q^{-1}z\big)P\big(q^{-1}z\big)-tR(qz)P(qz)\big)
\\ \qquad
{}=\int_0^1 T\big(R\big(q^{-1}{\rm e}^{2\pi {\rm i} x}\big)P\big(q^{-1}{\rm e}^{2\pi {\rm i} x}\big)-tR\big(q{\rm e}^{2\pi {\rm i} x}\big)P\big(q{\rm e}^{2\pi {\rm i} x}\big)\big)w(x)\,{\rm d}x.
\end{gather*}
Similarly to the proof of Theorem~\ref{ThrTwistedTraceAndQuasiPeriodic} we deduce that this is zero.

It is easy to see that different pairs of $w$ and~$l$ give different traces. Now we compute the dimension of the space of pairs. For convenience we denote by $\deg R$ the number of nonzero roots of a Laurent polynomial $R$ counted with multiplicities. The space of functionals~$l$ has dimension $\deg Q$. The space of functions $w$ has dimension $\deg Q+\deg P_*$ because we allow poles at roots of~$Q\big({\rm e}^{2\pi {\rm i} x}\big)$ and shifted roots of $P_*\big({\rm e}^{2\pi {\rm i} x}\big)$. So the overall dimension is $\deg P_*+2\deg Q=\deg P=n$ as we wanted.
\end{proof}

Now we turn to the general case.
Let $z$ be a complex number. If $\lvert z \rvert>\frac{1}{\lvert q\rvert}$ we find a minimal positive integer $k$ such that $\big\lvert q^{2k}z\big\rvert\leq \frac{1}{\lvert q\rvert}$ and denote $q^{2k}z$ by $\widetilde{z}$. If $\lvert z \rvert<\abs{q}$ we similarly find 	the smallest $k$ such that $\big|q^{-2k}z\big|\geq \lvert q\rvert$ and denote $q^{-2k}z$ by $\widetilde{z}$.

Suppose that $P$ has roots $\alpha_1,\dots,\alpha_k,\alpha_{k+1},\dots,\alpha_n$, where $\alpha_1,\dots,\alpha_k$ belong to $\ovl{U}=\big\{\abs{q}\leq\abs{x}\leq\frac{1}{\abs{q}}\big\}$ and $\alpha_{k+1},\dots,\alpha_n$ do not belong to $\ovl{U}$. We are counting roots with multiplicities, so some of $\alpha_i$ may be equal to each other. Let $\widetilde{P}$ be a polynomial with roots $\alpha_1,\dots,\alpha_k,\widetilde{\alpha_{k+1}},\dots,\widetilde{\alpha_n}$, $P_{\circ}$ be a polynomial with roots $\alpha_1,\dots,\alpha_k$.

\begin{thr}
\label{ThrTracesInGeneral}
Any trace $T$ can be represented as $T=\widetilde{T}+\Phi$, where $\Phi$ is a linear functional of the form
\[
\Phi(R)=\sum_{a\notin S^1,k\geq 0} c_{ak}R^{(k)}(a)
\]
and $\widetilde{T}$ is a trace for a $q$-deformation $\mc{A}_{\widetilde{P}}$ corresponding to $\widetilde{P}$. We will abbreviate this sentence to ``$\widetilde{T}$ is a trace for $\widetilde{P}$''. Furthermore, if $\Phi=0$ then $T$ is a trace for $P_{\circ}$.
\end{thr}
\begin{rem}
We note that the space of traces for $\widetilde{P}$ does not change when we multiply this polynomial by nonzero complex number, similarly for $P_{\circ}$.
\end{rem}
\begin{proof}
In this proof we will denote the coefficient of $z^l$ in a polynomial $R$ by $R_l$.
Consider $T'(S(x))=T\big(S\big(q^{-1}x\big)-tS(qx)\big)$. The map $T'$ is linear and satisfies $T'(P(x)R(x))\allowbreak=0$ for all $R(x)\in\CN\big[x,x^{-1}\big]$. If $t=q^{2l}$ for some $l\in\ZN$ we also have $T'\big(x^{-l}\big)=0$. From $T'\big(P(x)\CN\big[x,x^{-1}\big]\big)=\{0\}$ we get that $T'(S(x))=\sum_{i,j} c_{ij} S^{(j)}(\alpha_i)$ for some $c_{ij}\in\CN$. Here $\alpha_i$ are different roots of $P$ and $j$ is not bigger than the multiplicity of $\alpha_i$ as a root of~$P$.

Let $z$ be a complex number with $\lvert z\rvert>\frac{1}{\lvert q\rvert}$. Let $\widetilde{z}=q^{2k}z$. Let $R$, $S$ be polynomials such that $R(z)=S\big(q^{-1}z\big)-tS(qz)$. Then
\begin{equation*}
S(z)-t^k S(\widetilde{z})=\sum_{l=0}^{k-1} t^l \big(S\big(q^{2l}z\big)-tS\big(q^{2l+2}z\big)\big)
=\sum_{l=0}^{k-1} t^l R\big(q^{2l+1}z\big).
\end{equation*}
For $l=0,\dots,k-1$ we have $\big\lvert q^{2l+1}z\big\rvert>1$ by the definition of $\widetilde{z}$.

Similarly for $z$ with $\lvert z\rvert<\lvert q\rvert$, $\widetilde{z}=q^{-2k}z$ we have
\[
S(z)-t^{-k}S\big(q^{-2k}z\big)=-\sum_{l=0}^{k-1} t^{-l-1}R\big(q^{-1-2l}z\big).
\]
Differentiating we get similar formulas for $S^{(j)}$ and $R^{(j)}$:
\begin{equation}\label{EqSandRDerivatives}
S^{(j)}(z)-t^k q^{2kj} S^{(j)}(\widetilde{z})=\sum_{|w|\neq 1} s_w R^{(j)}(w).
\end{equation}

Let $\widetilde{\alpha_i}=z^{2k_i}\alpha_i$. Consider
\[
T_1(S(x))=\sum_{i,j} c_{ij} t^{k_i} q^{2k_i j}S^{(j)}(\widetilde{\alpha_i}).
\]
We have $T_1\big(S(x)\widetilde{P}(x)\big)=0$ for all $S(x)\in\CN\big[x,x^{-1}\big]$.
Using~\eqref{EqSandRDerivatives} we see that $T'(S)-T_1(S)$ depends only on
\[
R(x)=S\big(q^{-1}x\big)-tS(qx)
\]
and has the form $\sum_{i=1}^a\sum_{j=1}^b s_{z_i,j} R^{(j)}(z_i)$, where $\abs{z_i}\neq 1$ for all $i$. We denote this linear functional of $R$ by $\Phi(R)$.

If $t=q^{2l}$ for some integer~$l$ then $T_1\big(z^{-l}\big)=T'\big(z^{-l}\big)=0$.
For any polynomial $R$ with $R_{-l}=0$ there exists a polynomial $S$ such that $R(z)=S\big(q^{-1}z\big)\allowbreak-tS(qz)$. $S$ is defined up to adding $z^{-l}$. So we get $T$ from $T'$ as
\[
T(R)=R_{-l}T\big(z^{-l}\big)+T'\big(S\big(q^{-1}z\big)-tS(qz)\big),
\]
where
\[
S\big(q^{-1}z\big)-tS(qz)=R(z)-R_{-l}z^{-l}.
\]
It follows that
\[
\widetilde{T}(R)=R_{-l}T\big(z^{-l}\big)+T_1\big(S\big(q^{-1}z\big)-t S(qz)\big),
\]
where
\[
S\big(q^{-1}z\big)-tS(qz)=R(z)-R_{-l}z^{-l},
\]
is a well-defined linear map. Since $T_1$ is zero on~$\CN\big[z,z^{-1}\big]\widetilde{P}(x)$, the map $\widetilde{T}$ is a trace for $\widetilde{P}$.

In the case when there does not exist~$l$ such that $t=q^{2l}$, for any polynomial $R$ there exists a unique polynomial $S$ such that $R(z)=S\big(q^{-1}z\big)-tS(qz)$. In this case we define $\widetilde{T}(R)=T_1\big(S\big(q^{-1}z\big)-tS(qz)\big)$. The map $\widetilde{T}$ is a trace for $\widetilde{P}$.
So we get $T=\widetilde{T}+\Phi$. This proves the first statement of the theorem.

Suppose that $\Phi=0$. This linear functional was a linear combination of
\[
l_{z,j}(R)=\sum_{l=0}^{k-1} t^l \big(R\big(q^{2l+1}z\big)\big)^{(j)}
\]
for $\abs{z}>\frac{1}{\abs{q}}$ and similar functionals for $\abs{z}<\abs{q}$. Here $k$ is defined by $\widetilde{z}=q^{2k}z$.
It is easy to see that $l_{z,j}$ are linearly independent for different $z$, $j$. From $\Phi=0$ we deduce that in the sum $T'(S(x))=\sum_{i,j} c_{ij} S^{(j)}(\alpha_i)$ we have $c_{ij}=0$ for indices $i$ such that $\alpha_i\notin \ovl{U}$. So~$T'$ uses only roots of $P_{\circ}$. It~follows that $T$ is a trace for $P_{\circ}$.
\end{proof}

\section[Positivity of traces in the case when all roots of P satisfy q<|z|<q\textasciicircum{}\{-1\}]
{Positivity of traces in the case when all roots of $\boldsymbol P$\\ satisfy $\boldsymbol {q<\lvert z \rvert<q^{-1}}$}\label{SecAnnulus}
\subsection{Positivity for twisted traces via quasiperiodic functions}
In this section we assume that $0<q<1$.
Suppose that $\rho(u)=av$, $\rho(v)=bu$, $\rho(Z)=Z^{-1}$ is an antilinear automorphism of $\mc{A}$ such that $\rho^2=g_t$. Rescaling $v$ we may assume that $|a|=1$. We have $\rho^2(u)=\ovl{a}bu$, $\rho^2(v)=\ovl{b}av$ and~$\rho^2=g_t$, where $g_t(u)=tu$, $g_t(v)=t^{-1}v$. It~follows that $a\ovl{a}b\ovl{b}=1$, so $|b|=1$.
Since $t=\ovl{a}b$ we have $|t|=1$. Let $c\in [0,1)$ be a real number such that $t={\rm e}^{2\pi {\rm i} c}$.
If we change~$u$ to~$zu$, $\lvert z \rvert=1$, then $a$ changes to $\ovl{z}a=z^{-1}a$, $b$ changes to $z^{-1}b$. Therefore we can assume that $a={\rm e}^{-\pi {\rm i} c}$. It~follows that $b={\rm e}^{\pi {\rm i} c}$.
So we will assume that $\rho$ is an antilinear isomorphism such that $\rho(u)={\rm e}^{-\pi {\rm i} c}v$, $\rho(v)={\rm e}^{\pi {\rm i} c}u$, $\rho(Z)=Z^{-1}$, where $t={\rm e}^{2\pi {\rm i} c}$.
The conjugation $\rho$ exists when
\[
P\big(q^{-1}Z^{-1}\big)=\rho\big(P\big(q^{-1}Z\big)\big)=\rho(uv)=\rho(u)\rho(v)=vu=P(qZ).
\]
This is equivalent to $P(z)$ being real on $S^1$.

\begin{defn}
Let $T$ be a $g_t$-twisted trace. We say that $T$ is positive if the sesquilinear form $(a,b)_T=T(a\rho(b))$, $a,b\in\mc{A}$, is positive definite.
We note that any positive definite sesquilinear form is Hermitian.
\end{defn}
Recall that $\mc{A}=\oplus_{i\in\ZN}\mc{A}_i$, where for $a\in \mc{A}_i$ we have $ZaZ^{-1}=q^{2i} a$. Since $T$ is supported on~$\mc{A}_0$, the decomposition $\mc{A}=\oplus_{i\in \ZN}\mc{A}_i$ is orthogonal with respect to $(\cdot,\cdot)_T$. Therefore it is enough to check positive definiteness for homogeneous elements~$a$.
Recall that since $T$ is supported on $\mc{A}_0=\CN\big[Z,Z^{-1}\big]$ we write $T$ both for linear functional on~$\mc{A}_0$ and on $\CN\big[z,z^{-1}\big]$.

\begin{lem}\quad
\label{LemTwistedTracesAndPolynomials}
\begin{enumerate}\itemsep=0pt
\item[$1.$]
A $g_t$-twisted trace $T$ is positive if and only if $T(a\rho(a))>0$ for any nonzero $a\in \mc{A}_0\cup\mc{A}_1$.
\item[$2.$]
A $g_t$-twisted trace $T$ is positive if and only if
\[
T\big(R(z)\ovl{R}\big(z^{-1}\big)\big)>0\qquad
\text{and}\qquad
{\rm e}^{-\pi {\rm i} c}T\big(P\big(q^{-1}z\big)R\big(q^{-1}z\big)R\big(qz^{-1}\big)\big)>0
\]
for all nonzero Laurent polynomials $R$.
\end{enumerate}
\end{lem}

\begin{proof}
1.~Suppose that $(\cdot,\cdot)_T$ is positive definite on $\mc{A}_0$ and $\mc{A}_1$, $a$ is a homogeneous element of degree $i>1$. Then there exists an element $b$ of degree $0$ or $1$ and a positive integer $k$ such that $a=u^k b u^k$. We have
\[
\rho(a)=\rho(u)^k \rho(b) \rho(u)^k = {\rm e}^{-2k\pi {\rm i} c} v^k\rho(b)v^k=t^{-k} v^k \rho(b) v^k.
\]
Therefore
\begin{gather*}
T(a\rho(a))=t^{-k}T\big(u^k b u^k v^k \rho(b) v^k\big)=
t^{-k}T\big(b u^k v^k \rho(b) v^k g_t\big(u^k\big)\big)=T\big(b u^k v^k \rho(b) v^k u^k\big).
\end{gather*}
Let $s=b u^k v^k$. Note that
\[
\rho(s)=\rho\big(b u^k v^k\big)=\rho(b) {\rm e}^{-\pi {\rm i} k c} v^k {\rm e}^{\pi {\rm i} k c} u^k=\rho(b) v^k u^k.
\]
So we proved that $T(a\rho(a))=T(s\rho(s))$, where $s$ has degree $0$ or $1$. So $T(s\rho(s))>0$ and $T(a\rho(a))>0$.

If $a$ has negative degree then $a=\rho(b)$ for some $b$ homogeneous of positive degree. Then
\[
T(a\rho(a))=T(\rho(b)g_t(b))=T(b\rho(b))>0.
\]

2.~It is enough to take $a$ of degree $0$ or $1$. Let $a\in\mc{A}_0$. Then there exists $R\in\CN\big[z,z^{-1}\big]$ such that $a=R(Z)$. Therefore
\[
T(a\rho(a))=T\big(R(Z)\ovl{R}\big(Z^{-1}\big)\big)=T\big(R(z)\ovl{R}\big(z^{-1}\big)\big).
\]
Let $a\in\mc{A}_1$. Then there exists $R\in\CN\big[z,z^{-1}\big]$ such that $a=u R(qZ)$. Then
\begin{align*}
a\rho(a)&=u R(Z) {\rm e}^{-\pi {\rm i} c} v \ovl{R}\big(qZ^{-1}\big)
= {\rm e}^{-\pi {\rm i} c} R\big(q^{-1}Z\big) uv R\big(q Z^{-1}\big)
\\
&= {\rm e}^{-\pi {\rm i} c} R\big(q^{-1}Z\big) P\big(q^{-1}Z\big) R\big(q Z^{-1}\big).
\end{align*}
It follows that
\begin{gather*}
T(a\rho(a))={\rm e}^{-\pi {\rm i} c} T\big(R\big(q^{-1} z\big)P\big(q^{-1}z\big) R\big(qz^{-1}\big)\big).\tag*{\qed}
\end{gather*}
\renewcommand{\qed}{}
\end{proof}

Now we want to reformulate positivity conditions in terms of quasiperiodic functions in the case when all roots of $P$ belong to~$U$. Let $w$ be a nonzero function such that $w(x+1)=w(x)$, $w(x+\tau)=tw(x)$ and $w\big(x+\frac{\tau}{2}\big)P(x)$ is holomorphic on $\ovl{D}=[0,1]\times\big[{-}\frac{\tau}{2},\frac{\tau}{2}\big]$.

\begin{thr}\label{ThrPositivityThroughQuasiPeriodicFunctions}
Let $T(R(z))=\int_0^1 R\big({\rm e}^{2\pi {\rm i} x}\big) w(x)\,{\rm d}x$, where $w(x)$ is as above. Then $T$ is positive if and only if $w$ and ${\rm e}^{-\pi{\rm i}c}P\big({\rm e}^{2\pi {\rm i} x}\big)w\big(x+\frac{\tau}{2}\big)$ are nonnegative on $[0,1]$.
\end{thr}

\begin{proof}
We will use the following lemma in the proof.
\begin{lem}
\label{LemClosureOfPositivePolynomials}
The closure of the set $\big\{R\big({\rm e}^{2\pi {\rm i} x}\big)\ovl{R\big({\rm e}^{2\pi {\rm i} x}\big)}\big\}$ in $C[0,1]$ is the set of functions $f$ that are nonnegative on $[0,1]$ and satisfy $f(0)=f(1)$.
\end{lem}
\begin{proof}
Let $f\geq 0$, $f(0)=f(1)$. Then we can find $g\in C\big(S^1\big)$ such that $f=g\big({\rm e}^{2\pi {\rm i} x}\big)^2$. Since trigonometric polynomials are dense in $C\big(S^1\big)$, there exists a sequence of polynomials $R_n\in \CN\big[z,z^{-1}\big]$ such that $R_n\big({\rm e}^{2\pi {\rm i}x}\big)$ tends to $\sqrt{f}$ in $C[0,1]$. It~follows that $\ovl{R_n\big({\rm e}^{2\pi {\rm i} x}\big)}$ also tends to $\sqrt{f}$, so $R_n\big({\rm e}^{2\pi {\rm i} x}\big)\ovl{R_n\big({\rm e}^{2\pi {\rm i} x}\big)}$ tends to $f$.
\end{proof}

It follows from Lemma~\ref{LemTwistedTracesAndPolynomials} that $T$ is positive if and only if
\[
T\big(R(z)\ovl{R}\big(z^{-1}\big)\big)>0
\]
and
\[
{\rm e}^{-\pi {\rm i}c}T\big(P\big(q^{-1}z\big)R\big(q^{-1}z\big)\ovl{R}\big(qz^{-1}\big)\big)>0
\]
for all nonzero $R\in\CN\big[z,z^{-1}\big]$.
We have
\[
T\big(R(z)\ovl{R}\big(z^{-1}\big)\big)=\int_0^1 R\big({\rm e}^{2\pi {\rm i} x}\big)\ovl{R\big({\rm e}^{2\pi {\rm i} x}\big)}w(x)\,{\rm d}x.
\]
Using Lemma~\ref{LemClosureOfPositivePolynomials} we see that this expression is positive for all nonzero Laurent polynomials $R$ if and only if $w$ is nonnegative on $[0,1]$.
Recall that $q={\rm e}^{\pi {\rm i} \tau}$. We have
\begin{align*}
T\big(P\big(q^{-1}z\big)R\big(q^{-1}z\big)\ovl{R}\big(qz^{-1}\big)\big)&=
\int_0^1 P\big({\rm e}^{2\pi {\rm i}(x-\frac{\tau}{2})}\big)
R\big({\rm e}^{2\pi {\rm i} (x-\frac{\tau}{2})}\big)
\ovl{R}\big({\rm e}^{-2\pi {\rm i} (x-\frac{\tau}{2})}\big)w(x)\,{\rm d}x
\\
&=\int_{-\frac{\tau}{2}}^{1-\frac{\tau}{2}} P\big({\rm e}^{2 \pi {\rm i} x}\big)
R\big({\rm e}^{2\pi {\rm i} x}\big)\ovl{R}\big({\rm e}^{-2\pi {\rm i} x}\big)
w\left(x+\frac{\tau}{2}\right){\rm d}x
\\
&=\int_0^1 P\big({\rm e}^{2 \pi {\rm i} x}\big)R\big({\rm e}^{2\pi {\rm i} x}\big)
\ovl{R}\big({\rm e}^{-2\pi {\rm i} x}\big)w\left(x+\frac{\tau}{2}\right){\rm d}x.
\end{align*}
In the last equality we used that $P\big({\rm e}^{2\pi {\rm i} x}\big)w\big(x+\frac{\tau}{2}\big)$ is holomorphic on $\ovl{D}$.
Using Lemma~\ref{LemClosureOfPositivePolynomials} we see that
\[
{\rm e}^{-\pi {\rm i} c}T\big(P\big(q^{-1}z\big)R\big(q^{-1}z\big)\ovl{R}\big(qz^{-1}\big)\big)
\]
is positive for all nonzero Laurent polynomials $R$ if and only if
\[
{\rm e}^{-\pi {\rm i} c}P\big({\rm e}^{2\pi {\rm i} x}\big)w\left(x+\frac{\tau}{2}\right)
\]
is nonnegative on $[0,1]$.
\end{proof}

\subsection{Positivity conditions for quasiperiodic functions}
In this section we will describe the set of quasiperiodic functions $w$ that give positive traces.
Recall that $q={\rm e}^{\pi {\rm i}\tau}$, $U=\big\{\frac{1}{q}<|x|<q\big\}$, $V=\frac{1}{2\pi {\rm i}}\ln U=\RN\times \big({-}\frac{\tau}{2},\frac{\tau}{2}\big)$, $D$ is a fundamental region for lattice generated by $1,\tau$ such that $\ovl{D}$ is a parallelogram on vertices $-\frac{\tau}{2}$, $\frac{\tau}{2}$, $1+\frac{\tau}{2}$, $1-\frac{\tau}{2}$. Since $0<q<1$, we have $\tau\in \rmi\RN$, $\Im\tau>0$.
Recall that $g_t$-twisted traces $T$ are in one-to-one correspondence with quasiperiodic functions~$w$ such that $w(x+1)=w(x)$, $w(x+\tau)=tw(x)$ and $w\big(x+\frac{\tau}{2}\big)P\big({\rm e}^{2\pi {\rm i} x}\big)$ is holomorphic on~$\ovl{V}$. Denote the linear space of these functions by $L$.

Suppose that $w$ is a quasiperiodic function that is real on $\RN$. Since $w$ is meromorphic, we have $w(\ovl{z})=\ovl{w(z)}$. For $z=a-\frac{\tau}{2}$, $a\in\RN$ this gives{\samepage
\[
w(\ovl{z})=w\bigg(a+\frac{\tau}{2}\bigg)=\ovl{w\bigg(a-\frac{\tau}{2}\bigg)}
=\ovl{t}^{-1}\ovl{w\bigg(a+\frac{\tau}{2}\bigg)}=t w(\ovl{z}).
\]
Therefore ${\rm e}^{-\pi {\rm i} c}w\big(x+\frac{\tau}{2}\big)$ is real when $\Im z=\frac{1}{2}\Im\tau$.}

Denote the space of such functions by $L_{\RN}$. Theorem~\ref{ThrPositivityThroughQuasiPeriodicFunctions} says in particular that positive traces correspond to functions from $L_{\RN}$.

Since all roots of $P$ belong to $U$, there are $n$ roots of $P\big({\rm e}^{2\pi {\rm i} x}\big)$ in $D_0$, the interior of $D$. Denote them by $\alpha_1,\dots,\alpha_n$. Roots of $P$ are divided into pairs $z,\ovl{z}^{-1}$, and singletons $\lvert z \rvert=1$, so $\alpha_1,\dots,\alpha_n$ are in pairs $\alpha$, $\ovl{\alpha}$ and singletons $\alpha\in\RN$. Let $\beta_i=\alpha_i+\frac{1}{2}$. They satisfy the same symmetry condition as $\alpha_i$.

Recall that $\th(z)=\th(z;\tau)$ is the Jacobi theta function. Lemma~\ref{LemHowLLooks} says that
\[
L=\bigg\{\lambda\frac{\prod_{i=1}^n \th(z-a_i)}{\prod_{i=1}^n \th(z-\beta_i)}\bigg|\sum a_i-\sum\beta_i-c\in\ZN,\lambda\in\RN\bigg\}.
\]

We want to describe the behavior of $\th(x)$ and $\th\big(x+\frac{\tau}{2}\big)$ on the real line.

\begin{lem}\quad
\label{LemPositivityForThetaFunc}
 \begin{enumerate}\itemsep=0pt
 \item[$1.$]
 Suppose that $a\in \RN$. Then $\th(z-a)$ and
 \[
 \frac{{\rm e}^{\pi {\rm i}(z-a)}\th\big(z-a+\frac{\tau}{2}\big)}{\cos(\pi (z-a))}
 \]
 are positive on $\RN$.
 \item[$2.$]
 Suppose that $\Im a\notin \ZN{\rm i}\tau$. Then $\th(z-a)\th(z-\ovl{a})$ is nonnegative on $\RN$ and
 \[
 {\rm e}^{2\pi {\rm i}(z-\Re a)}\th\left(z-a+\frac{\tau}{2}\right)\th\left(z-\ovl{a}+\frac{\tau}{2}\right)
 \]
 is positive on $\RN$.
 \end{enumerate}
 \end{lem}

\begin{proof}
We will use Jacobi triple product identity
\begin{equation}
\label{EqProductForTheta}
\th(z)=\prod _{m=1}^{\infty }\big(1-{\rm e}^{2m\pi {\rm i}\tau}\big)\big(1+{\rm e}^{(2m-1)\pi {\rm i}\tau +2\pi {\rm i}z}\big)\big(1+{\rm e}^{(2m-1)\pi {\rm i}\tau -2\pi {\rm i}z}\big).
\end{equation}
It follows that
\begin{align}
\th\left(z+\frac{\tau}{2}\right)&=\prod _{m=1}^{\infty }\big(1-{\rm e}^{2m\pi {\rm i}\tau }\big)\big(1+{\rm e}^{2m\pi {\rm i}\tau +2\pi {\rm i}z}\big)\big(1+{\rm e}^{(2m-2)\pi {\rm i}\tau -2\pi {\rm i}z}\big)\nonumber
\\
&=\big(1+{\rm e}^{-2\pi {\rm i} z}\big)\prod_{m=1}^{\infty} \big(1-{\rm e}^{2m\pi {\rm i}\tau }\big)\big(1+{\rm e}^{2m\pi {\rm i}\tau +2\pi {\rm i}z}\big)\big(1+{\rm e}^{2m\pi {\rm i}\tau -2\pi {\rm i}z}\big).
\label{EqProductForThetaPlus}
\end{align}
We note that
\begin{equation}\label{EqConjugateOfExp}
\ovl{1+{\rm e}^{k\pi {\rm i} \tau\pm 2\pi {\rm i} z}}=1+{\rm e}^{k\pi {\rm i} \tau\mp 2\pi {\rm i} z}
\end{equation}
for all integers $k$. Also when $k\neq 0$, $z\in \RN$ we have $\big\lvert {\rm e}^{k\pi {\rm i} \tau\pm 2\pi {\rm i} z}\big\rvert = {\rm e}^{k\pi {\rm i}\tau}\neq 1$, so $1+{\rm e}^{k\pi {\rm i} \tau\pm 2\pi {\rm i} z}$ is nonzero when $z\in\RN$.

1. Comparing~\eqref{EqProductForTheta} and~\eqref{EqConjugateOfExp} we see that $\th(z)$ is a product of two nonzero conjugate numbers when $z\in \RN$, so it is positive.
Comparing~\eqref{EqProductForThetaPlus} and~\eqref{EqConjugateOfExp} we see that
\[
\big(1+{\rm e}^{-2\pi {\rm i} z}\big)^{-1}\th\left(z+\frac{\tau}{2}\right)
\]
is a product of two nonzero conjugate numbers when $z\in\RN$. We have{\samepage
\[
1+{\rm e}^{-2\pi {\rm i} z}=2{\rm e}^{-\pi {\rm i} z}\cos(\pi {\rm i} z).
\]
This proves the first part of the lemma.}

2. We have
\[
\big| {\rm e}^{k\pi {\rm i}\tau\pm 2\pi {\rm i} (z-a)}\big|=\big| {\rm e}^{k\pi {\rm i}\tau\pm 2\pi \Im a}\big|\neq 1
\]
for all real $z$ and even integers $k$ since $\Im a\notin \ZN{\rm i}\tau$. Therefore $1+{\rm e}^{k\pi {\rm i} \tau\pm 2\pi {\rm i} (z-a)}$ is nonzero when $z\in \RN$ and $k$ is even.
Comparing~\eqref{EqProductForTheta} and~\eqref{EqConjugateOfExp} we see that $\th(z-a)\th(z-\ovl{a})$ is a product of two conjugate numbers, so it is nonnegative when $z\in \RN$.
Comparing~\eqref{EqProductForThetaPlus} and~\eqref{EqConjugateOfExp} we see that
\[
\big(1+{\rm e}^{-2\pi {\rm i} (z-a)}\big)^{-1}\big(1+{\rm e}^{-2\pi {\rm i} (z-\ovl{a})}\big)^{-1}\th\left(z+\frac{\tau}{2}-a\right)\th\left(z+\frac{\tau}{2}-\ovl{a}\right)
\]
is a product of two nonzero conjugate numbers. Similarly to the above we have
\[
\big(1+{\rm e}^{-2\pi {\rm i} (z-a)}\big)\big(1+{\rm e}^{-2\pi {\rm i} (z-\ovl{a})}\big)=
4{\rm e}^{-2\pi {\rm i} (z-\Re a)}\cos(\pi (z-a))\cos(\pi(z-\ovl{a})).
\]
The second statement of the lemma follows.
\end{proof}

It follows from Lemma~\ref{LemPositivityForThetaFunc} that the denominator of
\[
w(x)=\lambda\frac{\th(z-a_1)\cdots\th(z-a_n)}{\th(z-\beta_1)\cdots\th(z-\beta_n)}
\]
is positive on the real line. One of the positivity conditions in Theorem~\ref{ThrPositivityThroughQuasiPeriodicFunctions} says that $w$ is nonnegative on $\RN$. In this case $\th(z-a_1)\cdots\th(z-a_n)$ does not change sign on $\RN$.
\begin{lem}
\label{LemDescriptionOfNonnegThetaProd}
Suppose that $f=\th(z-a_1)\cdots\th(z-a_n)$ does not change sign on $\RN$. Then there exist $\lambda\in\RN\setminus\{0\}$ and $b_1,\dots,b_n\in\CN$ divided into singletons $b_j\in\RN$ and pairs $b_j=\ovl{b_k}$ such that $f=\lambda\th(z-b_1)\cdots \th(z-b_n)$. Moreover, we can choose $b_j$ so that they satisfy $\abs{\Im b_j}<\Im \tau$.
 \end{lem}
 \begin{proof}
 Since $\ovl{f(z)}=f(\ovl{z})$ for any $j=1,\dots,n$ we have $f\big(\ovl{a_j}+\frac{1}{2}+\frac{\tau}{2}\big)=0$. So there exists $k$ such that $\ovl{a_j}-a_k\in\ZN+\ZN\tau$. Comparing multiplicities of roots, we can divide $a_1,\dots,a_n$ into singletons $a_j-\ovl{a_j}\in\ZN\tau$ and pairs $a_j-\ovl{a_k}\in \ZN+\ZN\tau$. Since $\th(z)=\th(z+1)$, we may assume that in pairs $a_j-\ovl{a_k}\in\ZN\tau$.
 Shifting all $a_j$ by some integer multiple of $\tau$, we get a set $d_1,\dots,d_n$ divided into singletons $d_j\in\RN\cup \big(\RN+\frac{\tau}{2}\big)$ and pairs $d_j=\ovl{d_k}$ such that
 \[
 f={\rm e}^{-\pi l z+r} \th(z-d_1)\cdots \th(z-d_n),
 \]
 where $l\in \ZN$, $r\in \CN$. Moreover, we may choose $d_j$ so that they satisfy $\abs{\Im d_j}<\abs{\Im \tau}$.
 Suppose that $d_j\in \RN+\frac{\tau}{2}$. Then $\th(z-d_j)$ has simple zeros on $\RN$. Since all roots of $f$ on $\RN$ have even multiplicity, there exists an even number of $k$ from $1$ to $n$ with $d_k=d_j$. Shifting half of $d_j\in\RN+\frac{\tau}{2}$ by $\tau$, we obtain a new sequence $b_1,\dots,b_n$ divided into singletons $b_j\in\RN$ and pairs $b_j=\ovl{b_k}$ such that $f={\rm e}^{-\pi {\rm i} m z+s}\th(z-b_1)\cdots \th(z-b_n)$, where $m\in\ZN$, $s\in\CN$.
 It follows from Lemma~\ref{LemPositivityForThetaFunc} that $\th(z)$ is positive on $\RN$. Therefore both $f$ and $\th(z-b_1)\cdots\allowbreak \th(z-b_n)$ are real on $\RN$ and do not change sign. So ${\rm e}^{-\pi {\rm i} m z+s}$ is also real on $\RN$ and does not change sign. Hence $m=0$ and ${\rm e}^{s}=\lambda\in\RN\setminus\{0\}$.
 \end{proof}

Before we describe the cone of positive traces we need to introduce an additional parameter that distinguishes the parameters $P$ and $-P$. Since $P$ and $-P$ give isomorphic algebras, the difference between $P$ and $-P$ is in the choice of another conjugation on the same algebra.

Note that
\[
\frac{P\big({\rm e}^{2\pi {\rm i} x}\big)}{\prodl_{j\colon\beta_j\in\RN} \cos(x-\beta_j)}
\]
is a continuous function that is real on the real line and has no roots on $\RN$, so it does not change sign on $\RN$. Let $m_0=1$ in the case when this fraction is negative on $\RN$ and $m_0=0$ in the case when this fraction is positive on $\RN$. In other words,
\[
(-1)^{m_0}\frac{P\big({\rm e}^{2\pi {\rm i} x}\big)}{\prodl_{j\colon\beta_j\in\RN} \cos(x-\beta_j)}
\]
is positive on $\RN$.

Now we are ready to describe the cone of functions that give a positive trace.

\begin{thr}
\label{ThrConeAsASetAnnulus}
The cone $C$ of positive $g_t$-twisted traces has dimension $n$. It~is isomorphic to the set of entire functions $f$ such that
\[
f(x+1)=f(x), \qquad f(x+\tau)={\rm e}^{-\pi {\rm i} (2n x+n\tau-2\sum\beta_j-2c)}f(x)
\]
and
\[
f(x), \qquad {\rm e}^{\pi {\rm i}(-nx+\sum\beta_j+m_0+c)}f\left(x+\frac{\tau}{2}\right)
\]
are nonnegative on $\RN$. At the level of sets $C$ consists of functions
\[
\lambda\frac{\th(z-a_1)\cdots\th(z-a_n)}{\th(z-\beta_1)\cdots\th(z-\beta_n)},
\]
where $\lambda>0$, $\sum a_i-\sum\beta_i-c-m_0\in 2\ZN$ and all $a_i$ are divided into pairs $(a_i,a_j)$ with $a_i=\ovl{a_j}$. In particular, if $Q$ is another Laurent polynomial with the same number of nonzero roots $n$ counted with multiplicities, $s$ is a nonzero complex number with $\abs{s}=1$, cones $C_{P,t}$ and $C_{Q,s}$ are isomorphic.
\end{thr}
\begin{proof}
A quasiperiodic function $w\in L$ gives a positive trace if and only if $w$ and ${\rm e}^{-\pi {\rm i} c}P\big({\rm e}^{2\pi {\rm i} x}\big)\times w\big(x+\frac{\tau}{2}\big)$ are nonnegative on $\RN$.
We know that
\[
w=\lambda\frac{\prod_{i=1}^n \th(z-a_i)}{\prod_{i=1}^n \th(z-\beta_i)},
\]
where $\sum a_i-\sum \beta_i-c\in\ZN$, $\lambda\in\RN^{\times}$.
It follows from Lemma~\ref{LemPositivityForThetaFunc} that $\prod_{i=1}^n \th(z-\beta_i)$ is positive on~$\RN$. Since $w$ is nonnegative on~$\RN$ we deduce that $\prod_{i=1}^n \th(z-a_i)$ does not change sign on~$\RN$. Using Lemma~\ref{LemDescriptionOfNonnegThetaProd} we may assume that $a_1,\dots,a_n$ are divided into singletons $a_i\in\RN$ and pairs $a_i=\ovl{a_j}$. Since $w$ is nonnegative on~$\RN$ we get that $\lambda>0$.

On the other hand, suppose that $\lambda>0$,
\[
w=\lambda\frac{\prod_{i=1}^n \th(z-a_i)}{\prod_{i=1}^n \th(z-\beta_i)}
\]
is an element of $L$ and $a_1,\dots,a_n$ are symmetric with respect to $\RN$. It~follows from Lemma~\ref{LemPositivityForThetaFunc} that $\th(z-a_i)$ is positive on $\RN$ when $a_i\in \RN$. We deduce that in this case $w$ is nonnegative on $\RN$.

We now study the behavior of $w\big(x+\frac{\tau}{2}\big)$ on $\RN$. We say that two meromorphic functions $f$ and~$g$ are equivalent if $\frac{f}{g}$ is a positive function when restricted to $\RN$.
For all $i$ we have $|\Im\beta_i|<\Im\tau$ and $|\Im a_i|<\Im\tau$.
It follows from Lemma~\ref{LemPositivityForThetaFunc} that $\prod_{i=1}^n \th\big(z-\beta_i+\frac{\tau}{2}\big)$ is equivalent to
\[
{\rm e}^{-\pi {\rm i}(nz-\sum\beta_j)}\prod_{j\colon\beta_j\in\RN}\cos(\pi(z-\beta_j))
\]
and $\prod_{i=1}^n \th\big(z-a_i+\frac{\tau}{2}\big)$ is equivalent to
\[
{\rm e}^{-\pi {\rm i}(nz-\sum a_j)}\prod_{j\colon a_j\in\RN}\cos(\pi(z-a_j)).
\]
Let $\sum a_i-\sum\beta_j=c+m$, where $m\in\ZN$.
Therefore $w(z+\frac{\tau}{2})$ is equivalent to
\[
{\rm e}^{\pi {\rm i} c+\pi {\rm i} m}\frac{\prod_{a_j\in\RN}\cos(\pi(z-a_j))}{\prod_{\beta_j\in\RN}\cos(\pi(z-\beta_j))}.
\]
Recall that $P(x)$ is equivalent to \((-1)^{m_0}\prod_{\beta_j\in\RN}\cos(x-\beta_j)$.
It follows that ${\rm e}^{-\pi {\rm i} c}P\big({\rm e}^{2\pi {\rm i} x}\big)w\big(x+\frac{\tau}{2}\big)$ is equivalent to
\[
(-1)^{m+m_0}\prod_{a_j\in \RN} \cos(\pi(z-a_j)).
\]
This function is nonnegative on $\RN$ if and only if each real number appears even number of times among $a_j$ and $m$ has the same parity as $m_0$. This gives a description of $C$ as a set.

Now we describe $C$ in terms of the numerator
\[
f(z)=\th(z-a_1)\cdots \th(z-a_n).
\]
The function
\[
w(z)=\frac{f(z)}{\th(z-\beta_1)\cdots(z-\beta_n)}
\]
belongs to $C$ only if $w(x)$ and ${\rm e}^{-\pi {\rm i} c}w\big(x+\frac{\tau}{2}\big)P\big({\rm e}^{2\pi {\rm i} x}\big)$ are nonnegative on $\RN$. It~follows from Lemma~\ref{LemDescriptionOfNonnegThetaProd} that $w(x)$ is equivalent to $f(x)$ on $\RN$. Recall that $\prod_{i=1}^n \th\big(z-\beta_i+\frac{\tau}{2}\big)$ is equivalent to
\[
\exp\Big({-}\pi {\rm i}\Big(nz-\sum\beta_j\Big)\Big)\prod_{\beta_j\in\RN}\cos(\pi(z-\beta_j)).
\]
It follows that $w\big(x+\frac{\tau}{2}\big)P\big({\rm e}^{2\pi {\rm i} x}\big)$ is equivalent to
\[
(-1)^{m_0}f\left(x+\frac{\tau}{2}\right){\rm e}^{-\pi {\rm i}\left(nz-\sum \beta_j\right)}.
\]
So ${\rm e}^{-\pi {\rm i} c}w\big(x+\frac{\tau}{2}\big)P\big({\rm e}^{2\pi {\rm i} x}\big)$ is nonnegative on $\RN$ if and only if
\[
f\left(x+\frac{\tau}{2}\right){\rm e}^{-\pi {\rm i}\left(nz+m_0+c-\sum \beta_j\right)}
\]
is nonnegative on $\RN$.
We see from this statement that there exists an isomorphism between any two cones $C_{P,t}$ and~$C_{Q,s}$ given by the map $f(x)\mapsto f(x+x_0)$ for $x_0=\sum\beta_{P,i}+c_t+m_{P,0}-\sum\beta_{Q,i}-c_s-m_{Q,0}\in\RN$.
We deduce the statement about dimension from the example below.
\end{proof}
\begin{Example}
Suppose that $P$ is positive on $S^1$, $t=1$. Then the space $L$ consists of elliptic functions $w$ such that $P\big({\rm e}^{2\pi {\rm i} x}\big)w\big(x+\frac{\tau}{2}\big)$ is holomorphic when $\abs{\Im x}\leq \frac12\abs{\tau}$. In particular, $1$~is an element of $L$. The linear map $\phi\colon C\to L_{\RN}/(\RN\cdot 1)$ is surjective. Its fibers are isomorphic to~$\RN_{\geq 0}$ at nonzero points and to~$\RN_{>0}$ at zero.
\end{Example}
\begin{proof}
By definition $L$ consists of functions $w$ such that $w(x+1)=w(x)$, $w(x+\tau)=tw(x)$ and $P\big({\rm e}^{2\pi {\rm i} x}\big)w\big(x+\frac{\tau}{2}\big)$ is holomorphic when $\abs{\Im x}\leq \frac 12\abs{\tau}$, $\abs{\Re x}\leq \frac12$. Since $P\big({\rm e}^{2\pi {\rm i} x}\big)w\big(x+\frac{\tau}2\big)$ is $1$-periodic it is holomorphic in the closed strip $\abs{\Im x}\leq\frac12\abs{\tau}$.
When $t=1$ we get that $w$ is an elliptic function.
Theorem~\ref{ThrPositivityThroughQuasiPeriodicFunctions} says that $w\in C$ if and only if $w(x)$ and $w\big(x+\frac{\tau}{2}\big)P\big({\rm e}^{2\pi {\rm i} x}\big)$ are nonnegative on~$\RN$. Since $P$ is positive on $S^1$, this means that $w$ is nonnegative on $\RN\cup \big(\RN+\frac 12\tau\big)$.
Since $P$ has no roots on $S^1$, functions from $L$ do not have poles on $\RN+\frac 12\tau$. It~follows that for any nonconstant $w\in L$ we have
\[
(w+\RN\cdot 1)\cap C=\Big\{w+c\mid c+\min_{\RN\cup\left(\RN+\frac 12\tau\right)}w\geq 0\Big\}.
\]
This proves the remaining statements.
\end{proof}
\begin{Example}
Let $n=2$. In this case $C$ is isomorphic to $\RN_{\geq 0}^2\setminus\{0\}$.
Indeed, extremal points of $C$ are given by functions that have roots on $\RN$ or $\RN+\frac{\tau}{2}$. Since they don't change sign on $\RN$ or $\RN+\frac{\tau}{2}$, these roots are repeated. We have
\[
w=\lambda\frac{\th(z-a_1)\th(z-a_2)}{\th(z-\beta_1)\th(z-\beta_2)},
\]
where $a_1+a_2\in x_0+\ZN$ for some $x_0\in\RN$. Since the roots are repeated we have $a_1=a_2$, which gives two options: $a_1=a_2=\frac{x_0}{2}$ and $a_1=a_2=\frac{x_0+1}{2}$. We deduce that $C$ is generated by two linearly independent elements. Hence $C$ is isomorphic to $\RN_{\geq 0}^2\setminus\{0\}$.
\end{Example}

\section{Positivity of traces in the general case}\label{SecNotAnnulus}
Now we describe the cone of positive traces in the general case. The argument here is very similar to~\cite[Sections 4.3--4.4]{EKRS}, so we have taken some proofs from that article with necessary modifications. In this section we assume that $0<q<1$.

\subsection[The case when all roots of P satisfy q<=|z|<= q\textasciicircum{}\{-1\}]
{The case when all roots of $\boldsymbol P$ satisfy $\boldsymbol {q\leqslant \lvert z \rvert\leqslant q^{-1}}$}

Let $P(x)=P_*(x)Q(qx)Q\big(qx^{-1}\big)$, where all roots of $P_*$ belong to $U$, all roots of $Q$ belong to $S^1$. For any $R\in\CN\big[x,x^{-1}\big]$ let $R(x)=R_1Q(x)+R_0(x)$, where we take $R_0$ from some fixed linear space of representatives modulo $Q$.
Proposition~\ref{PropTracesClosedAnnulus} says that in this case any trace $T$ can be written as
\[
T(R)=\int_0^1 R_1\big({\rm e}^{2\pi {\rm i} x}\big)Q\big({\rm e}^{2\pi {\rm i} x}\big)w(x)\,{\rm d}x+l(R_0),
\]
where $w$ is a quasiperiodic function such that $Q\big({\rm e}^{2\pi {\rm i} x}\big)w(x)$ is holomorphic on $[0,1]$ and $P\big({\rm e}^{2\pi {\rm i} x}\big)\allowbreak\times w\big(x+\frac{\tau}{2}\big)$ is holomorphic on $V$,~$l$ is an arbitrary linear functional.

Suppose that $T$ is a positive trace as above. First we will prove that $w$ satisfies the same positivity conditions as in Theorem~\ref{ThrPositivityThroughQuasiPeriodicFunctions}. Then we will prove that $w$ has no poles. This means that on the level of quasiperiodic functions we get the same cone as before. The space of linear functionals multiplies this cone by $\RN_{\geq 0}^r$, where $r$ is the number of distinct roots of $Q$.

\begin{prop}
\label{PropClosedAnnulusWIsNonnegative}
Suppose that $T$ is a positive trace as above. Then $w$ and $w\big(x+\frac{\tau}{2}\big)P\big({\rm e}^{2\pi {\rm i} x}\big)$ are nonnegative on $\RN$.
\end{prop}
\begin{proof}
Since $T$ is positive we deduce from Lemma~\ref{LemTwistedTracesAndPolynomials} that
\[
T\big(R_1(z)\ovl{R_1}\big(z^{-1}\big)Q(z)\ovl{Q}\big(z^{-1}\big)\big)>0
\]
and
\[
{\rm e}^{-\pi {\rm i} c}T\big(P\big(q^{-1}z\big)R_1\big(q^{-1}z\big)R_q\big(qz^{-1}\big)\big)>0
\]
for any nonzero $R_1\in\CN\big[z,z^{-1}\big]$. These polynomials are divisible by $Q$. So in this case the $l(R_0)$ term is zero and $T$ is given by integration. Now we deduce from Lemma~\ref{LemClosureOfPositivePolynomials} similarly to the proof of Theorem~\ref{ThrPositivityThroughQuasiPeriodicFunctions} that $w(x)Q\big({\rm e}^{2\pi {\rm i} x}\big)\ovl{Q\big({\rm e}^{2\pi {\rm i} x}\big)}$ and $w\big(x+\frac{\tau}{2}\big)P\big({\rm e}^{2\pi {\rm i} x}\big)$ are nonnegative on $\RN$. It~follows that $w(x)$ is nonnegative on $\RN$.
\end{proof}

\begin{prop}\label{PropNoPolesClosedStrip}
Let $T$ be a trace as above. If $w$ has poles on $[0,1]$ then $T$ does not give a~posi\-tive definite form.
\end{prop}
\begin{proof}
Since $\ovl{R}\big(z^{-1}\big)=\ovl{R(z)}$ for $z\in S^1$, we will denote $\ovl{R}\big(z^{-1}\big)$ by $\ovl{R}$ in this proof.
Suppose that $w$ has poles on $\RN$ and $T$ is positive. We deduce that $w$ is nonnegative on $\RN$, so all poles of $w$ have even order. Suppose that $w$ has a pole of order $N\geq 2$ at a point $x_0$.
Let $f$ be an element of $L^2\big(S^1\big)$, $R_n$ be a sequence of polynomials that tends to $f$ in $L^2\big(S^1\big)$. Let $b$ be a real number. Consider $S_n=(R_nQ+b)\ovl{(R_nQ+b)}=Q\big(R_n\ovl{R_n}\ovl{Q}+b\ovl{R_n}+bR_n\frac{\ovl{Q}}{Q}\big)+b^2$. We have
\[
T(S_n)=\int_0^1 w\cdot Q\cdot \bigg(R_n\ovl{R_n}\ovl{Q}+b\ovl{R_n}+bR_n\frac{\ovl{Q}}{Q}\bigg){\rm d}x+Cb^2,
\]
where $C$ is some constant. If $C$ is negative then the statement is clear, so we assume that $C\geq 0$.
Since $R_n$ tends to $f$, $\ovl{R_n}$ tends to $\ovl{f}$ in $L^2\big(S^1\big)$ and $R_n\ovl{R_n}$ tends to $f\ovl{f}$ in $L^1\big(S^1\big)$. We deduce that $T(S_n)$ tends to
\[
\int_0^1 w\cdot Q\cdot \bigg(f\ovl{f}\ovl{Q}+b\ovl{f}+bf\frac{\ovl{Q}}{Q}\bigg)+Cb^2=\int_0^1 w\big(Q\ovl{Q}f\ovl{f}+b\ovl{f}Q+bf\ovl{Q}\big){\rm d}x+Cb^2.
\]
Let $Q$ have a zero of order $M$ at point $x_0$. Then $wQ\ovl{Q}$ has a zero of order $2M-N$ at $x_0$, $w\ovl{Q}$ and $wQ$ have a zero of order $M-N$ at $x_0$.
Let $a$ be a complex number, $\eps>0$. We define $f$ so that $f\big({\rm e}^{2\pi {\rm i} x}\big)=a\chi_{[x_0-\eps,x_0+\eps]}(x)$ for $x\in [0,1]$. When $x_0=0$ we set $f\big({\rm e}^{2\pi {\rm i} x}\big)=a\chi_{[x_0-\eps,x_0+\eps]}(x)$ for $x\in \big[{-}\frac12,\frac 12\big]$.
We deduce that
\[
\int_0^1 f\ovl{f}Q\ovl{Q}w\, {\rm d}x=c_1\eps^{2N-2M+1},\qquad
\int_0^1 \ovl{f}Q w\, {\rm d}x =c_2\eps^{N-M+1},
\]
where $c_1=c_1(\eps)$ has strictly positive limit at $\eps=0$, $c_2=c_2(\eps)$ has nonzero limit at $\eps=0$. We~choose $a$ so that $c_2$ has strictly positive limit at $\eps=0$.
We deduce that
\[
\lim_{n\to\infty}T(S_n)= c_1\eps^{2M-N+1}+2c_2\eps^{M-N+1}b+Cb^2.
\]
This is a quadratic polynomial in $b$ with discriminant
\[
D=4c_2^2\eps{2M-2N+2}-4Cc_1\eps^{2M-N+1}=\eps^{2M-2N+2}\big(4c_2^2-4Cc_1\eps^{N-1}\big).
\]
Since $N\geq 2$ for small $\eps$ this discriminant is positive. So there exists $b$ such that $\lim_{n\to\infty} T(S_n)\allowbreak<0$. Therefore $T\big((R_nQ+b)\ovl{(R_nQ+b)}\big)=T(S_n)<0$ for some $n$, a contradiction.
\end{proof}

Now we are left with the case when $w$ has no poles on $\RN$. In this case $T(R(z))=\int_0^1 R\big({\rm e}^{2\pi {\rm i} x}\big)\times w(x)\,{\rm d}x+\eta(R_0)$, where $\eta$ is some linear functional.
\begin{prop}\label{PropEtaIsSumOfValues}
Let $T$ be a trace as above. Then $T$ gives a positive definite form if and only if $\eta(R_0)=\sum c_j R_0(z_j)$, where $c_j\geq 0$, $z_j\in S^1$ are roots of $Q$.
\end{prop}

\begin{proof}
Suppose that this is not the case. Then it is easy to find a polynomial $S$ such that $\eta\big(\big(S\ovl{S}\big)_0\big)<0$. Consider a sequence of continuous functions $f_n$ such that $f_nQ+S$ tends to zero in $L^2\big(S^1\big)$. Since polynomials are dense in $C\big(S^1\big)$, we can find a sequence of polynomials $G_n$ such that $G_nQ+S$ tends to zero in $L^2\big(S^1\big)$. We deduce that
\[
T\big((G_nQ+S)\ovl{(G_nQ+S)}\big)\to \eta\big(\big(S\ovl{S}\big)_0\big)<0.
\]
This gives a contradiction.
\end{proof}

\subsection{General case}
Now we deal with the general case. Let $\widetilde{P}$, $P_{\circ}$ be polynomials defined in Section~\ref{SubSecTracesInGeneral}. The roots of $\widetilde{P}$ and $P_{\circ}$ belong to the set $\ovl{U}=\big\{\abs{q}\leq\abs{x}\leq\frac{1}{\abs{q}}\big\}$. The roots of $\widetilde{P}$ are obtained from the roots of $P$ by multiplication by $q^{2k}$ with minimal possible $|k|$. The multiplicity of $\alpha\in\ovl{U}$ as a~root of~$P_{\circ}$ equals to the multiplicity of $\alpha$ as a root of~$P$.
Theorem~\ref{ThrTracesInGeneral} says that in this case any trace~$T$ can be represented as $T=\widetilde{T}+\Phi$, where $\widetilde{T}$ is a trace for $\widetilde{P}$ and $\Phi$ is a linear functional such that
\[
\Phi(R)=\sum_{j=1}^m \sum_k c_{jk} R^{(k)}(z_j),
\]
where $z_j\notin S^1$. Furthermore, if $\Phi=0$ then $T$ is a trace for $P_{\circ}$.

\begin{prop}\label{PropPositiveThenPhiZero}
Let $T$ be a trace such that $\Phi\neq 0$. Then $T$ does not give a positive definite form.
\end{prop}

\begin{proof}
We denote $\ovl{S}\big(z^{-1}\big)$ by $\ovl{S}$.
For big enough $k$ we have $\Phi\big((z-z_1)^k\cdots (z-z_m)^k\CN\big[z,z^{-1}\big]\big)=\{0\}$. Recall that there exists a~polynomial $Q$ such that for any linear space $S$ of representatives modulo $Q$ and any polynomial $R=R_1Q+R_0$, $R_0\in S$, we have $\widetilde{T}(R)=\int_0^1 R_1 Q w\, {\rm d}x+\psi(R_0)$, where $w$ is a function and $\psi$ is some linear functional. Moreover, all roots of~$Q$ belong to~$S^1$. We can still represent $\widetilde{T}$ in this form if we change $Q$ to $Q_*=Q\ovl{Q}$. So we may assume that~$Q$ is nonnegative on~$S^1$. Let~$U$ be a~polynomial divisible by~$Q$ and real on $S^1$ such that $\Phi\big(U(z)\CN\big[z,z^{-1}\big]\big)=\{0\}$.

Let $F$ be any polynomial. It~is easy to find a sequence of continuous functions $f_n$ such that $f_nU-F$ tends to zero in $L^2\big(S^1\big)$. Approximating $f_n$ with polynomials, we find a sequence of polynomials $R_n$ such that $UR_n-F$ tends to zero in $L^2\big(S^1\big)$. It~follows that $H_n=(UR_n-F)\times\ovl{(UR_n-F)}$ tends to zero in $L^1\big(S^1\big)$. We deduce that $\widetilde{T}(H_nQ)=\int_0^1 H_n Q\cdot w\, {\rm d}x$ tends to zero when~$n$ tends to infinity.
It follows that $T(Q(z)H_n(z))$ tends to $\Phi(Q(z)H_n(z))=\Phi\big(Q(z)F(z)\ovl{F}\big(z^{-1}\big)\big)$. Since $QH_n$ is nonnegative on $S^1$, we have $T(QH_n)>0$. Now we get a~contradiction with
\begin{lem}
There exists $F(z)\in\CN\big[z,z^{-1}\big]$ such that
\[
\Phi\big(Q(z)F(z)\ovl{F}\big(z^{-1}\big)\big)\notin \RN_{\geq 0}.
\]
\end{lem}

\begin{proof}Let $r$ be the biggest number such that there exists $j$ with $c_{jr}\neq 0$. Fix this $j$. If there exists $z_s=\ovl{z_j}^{-1}$, we denote this index $s$ by $j^*$. Let
\[
F(z)=G(z)(z-z_1)^{r+1}\cdots (z-z_j)^r\cdots \widehat{(z-z_{j^*})} \cdots (z-z_m)^{r+1}.
\]
Notation $\widehat{(z-z_{j^*})} $ means that we omit $z-z_{j^*}$ in this product.
We note that
\[
c_{ik}\big(Q(z)F(z)\ovl{F}\big(z^{-1}\big)\big)^{(k)}|_{z=z_i}=0
\] for all $i$, $k$ except $k=r$ and $i=j$ or $i=j^*$. It~follows that
\begin{align*}
\Phi\big(Q(z)F(z)\ovl{F}\big(z^{-1}\big)\big)&=
c_{jk}\big(Q(z)F(z)\ovl{F}\big(z^{-1}\big)\big)^{(k)}(z_j)+c_{j^*k}\big(Q(z)F(z)\ovl{F}\big(z^{-1}\big)\big)(z_{j^*})
\\
&=c_{jr}Q(z_j)F^{(k)}(z_j)\ovl{F}\big(z_j^{-1}\big)+c_{j^*k}Q(z_{j^*})
F(z_{j^*})(-z_{j^*})^{-2k}\ovl{F}^{(k)}(z_{j^*})
\\
&=sa+t\ovl{a},
\end{align*}
 where $a=Q(z_j)F^{(k)}(z_j)\ovl{F}\big(z_j^{-1}\big)$, $s,t\in \CN$ are not simultaneously zero. Pick $a\in \CN$ so that $sa+t\ovl{a}\notin \RN_{\geq 0}$ and choose $G\in \CN[z]$ which gives this value of $a$. For example, we can choose $G$ linear. Then $\Phi\big(Q(z)F(z)\ovl{F}\big(z^{-1}\big)\big)\notin \RN_{\geq 0}$, as desired.
\end{proof}\renewcommand{\qed}{}
\end{proof}

Consider a $q$-deformation $\mc{A}_{\circ}$ with parameter $P_{\circ}$. We choose $P_{\circ}$ so that $\frac{P}{P_{\circ}}$ is positive on $S^1$.
\begin{cor}
The cone of positive definite forms for $\mc{A}$ coincides with the cone of positive definite forms for $\mc{A}_{\circ}$. Namely, $T\colon\CN\big[z,z^{-1}\big]\to \CN$ gives a positive definite form for $\mc{A}$ if and only if it gives a positive definite form for $\mc{A}_{\circ}$.
\end{cor}
\begin{proof}
Let $T=\widetilde{T}+\Phi$ be a positive trace. Using Proposition~\ref{PropPositiveThenPhiZero} we deduce that $\Phi=0$. Therefore $T$ is a trace for $P_{\circ}$. Writing $P_{\circ}(z)=P_*(z)Q(z)\ovl{Q}\big(z^{-1}\big)$ and choosing a linear space of representatives $S$ modulo $Q$ we can write $T(R)=\int_0^1 R_1Qw\, {\rm d}x+l(R_0)$, where $R=R_1Q+R_0$, $R_0\in S$.
Using Propositions~\ref{PropClosedAnnulusWIsNonnegative} and~\ref{PropNoPolesClosedStrip} we deduce that $w$ has no poles on $\RN$ and that $w(x)$ and \mbox{$w\big(x+\frac{\tau}{2}\big)P\big({\rm e}^{2\pi {\rm i} x}\big)$} are nonnegative on $\RN$. Using Proposition~\ref{PropEtaIsSumOfValues} we get that $T(R)=\int_0^1 R\big({\rm e}^{2\pi {\rm i} x}\big)\times w(x)\,{\rm d}x+\eta(R),$ where $\eta(R)=\sum c_j R(z_j)$, $c_j\geq 0$ and $z_j$ is a root of $Q$. On the other hand, from such $w$ and $\eta$ we obtain a positive trace for $P$.
Using this argument for $P_{\circ}$ instead of $P$ we get that $T$ is positive for $\mc{A}_{\circ}$ if and only if $T=\int_0^1 R w\, {\rm d}x+\eta$, where $w(x)$ and $w\big(x+\frac{\tau}{2}\big)P_{\circ}\big({\rm e}^{2\pi {\rm i} x}\big)$ are nonnegative on $\RN$, $\eta(R)=\sum c_j R(z_j)$, $c_j\geq 0$, $z_j$ is a root of~$Q$.
Since $\frac{P}{P_{\circ}}$ is positive on $S^1$, the function $w\big(x+\frac{\tau}{2}\big)P(x)$ is nonnegative on $\RN$ if and only if $w\big(x+\frac{\tau}{2}\big)P_{\circ}(x)$ is nonnegative on~$\RN$. It~follows that $T$ is positive for $P$ if and only if it is positive for~$P_{\circ}$.
\end{proof}

So we have proved the following theorem.

\begin{thr}\label{ThrConeInGeneralCase}
Let $\mc{A}$ be a $q$-deformation with parameter $P$. Let $t\in S^1$, $\rho_t$ be the corresponding conjugation. Let~$l$ be the number of roots $\alpha$ of $P$ such that $q<|\alpha|<q^{-1}$, $r$ be the number of distinct roots $\alpha$ with $|\alpha|=q$. Then the cone $\mc{C}_+$ of positive definite $\rho_t$-equivariant traces is isomorphic to $\mc{C}_l\times \RN_{\geq 0}^r$, where $\mc{C}_l$ consists of nonzero entire functions $f$ such that
\begin{enumerate}\itemsep=0pt
\item[$(1)$]
$f(x+1)=f(x)$, $f(x+\tau)={\rm e}^{-\pi {\rm i} n (\tau+2x)} f(x)$,
\item[$(2)$]
$f(x)$ and ${\rm e}^{\pi {\rm i} n x}f\left(x+\frac{\tau}{2}\right)$ are nonnegative on $\RN$.
\end{enumerate}
\end{thr}

\section{Short star-products}\label{SecStarProducts}
\subsection{Classification of deformations and automorphisms}

Let~$A$ be an arbitrary $\ZN_{\geq 0}$-graded algebra such that $A_0=\CN$. A \textit{short star-product}~\cite{ES} on~$A$ is an associative product $*$ on~$A$ such that for all nonnegative integers $k$, $m$ and $a\in A_k$, $b\in A_m$ we have
\[
a*b=c_0(a,b)+c_{2}(a,b)+c_{4}(a,b)+\cdots,
\]
where $c_0(a,b)=ab$, $c_i(a,b)\in A_{k+m-2i}$, $c_i$ is a bilinear map and $c_i=0$ for $i>\min(k,m)$. A~short star-product is nondegenerate if the form $(\cdot,\cdot)\colon A_k\times A_k\to\CN$, $(a,b)=c_k(a,b)=CT(a*b)$ is nondegenerate for all $k$. Here $CT$ means taking constant term.

We see that~$A$ with the multiplication given by $*$ is a filtered algebra and associate graded algebra of $(A,*)$ is naturally isomorphic to~$A$. Consider the operator $s\colon A\to A$ that acts on~$A_i$ as $(-1)^i$. It~follows from definitions that $s$ is an automorphism of~$A$ and an automorphism of~$(A,*)$. We say that a filtered algebra $\mc{A}$ equipped with an involution $S$ is a $\ZN/2\ZN$-equivariant deformation of~$A$ if $\gr S=s$. We see that $(A,*)$ is a $\ZN/2\ZN$-equivariant deformation of~$A$.

In~\cite{ES} it is proved that nondegenerate short star-products correspond to triples $(\mc{A},g,T)$. Here $\mc{A}$ is a $\ZN/2\ZN$-equivariant deformation of~$A$, $g$ is a filtered automorphism of $\mc{A}$ and $T$ is a~nondegenerate $g$-twisted trace. A trace is called nondegenerate if the form $(\cdot,\cdot)\colon\mc{A}_{\leq i}\times\mc{A}_{\leq i}\to\CN$ defined by $(a,b)=T(ab)$ is nondegenerate for all $i\geq 0$.

In~\cite{ES} is assumed that~$A$ is commutative and equipped with Poisson bracket, but the proof should be the same. In this paper we need this correspondence only as a motivation to classify deformations and study nondegenerate traces, so we will not prove this result for noncommutative graded algebras~$A$. We will prove this result in a new paper~\cite{K2}.

Let $\mc{A}$ be a generalized $q$-Weyl algebra.
In this section we assume that $P$ belongs to $\CN[z]$ and has degree~$n$. We no longer assume that $q$ is positive, but we assume that $0<\abs{q}<1$.
Consider the subalgebra $\mc{A}_+$ of $\mc{A}$ generated by~$u$,~$v$,~$Z$. There is a filtration on $\mc{A}_+$ defined by $\deg u=\deg v=n$, $\deg Z=2$. We have $\gr\mc{A}_+=A_+$, where $A_+$ is generated by~$u$,~$v$,~$Z$ with relations $Zu=q^2uZ$, $Zv=q^{-2}vZ$, $uv=q^{-n}Z^{n}$, $vu=q^nZ^n$, $\deg u=\deg v=n$, $\deg Z=2$.
We will do the following in this section. First, we will prove that any deformation~$\mc{B}$ of~$A_+$ is isomorphic to $\mc{A}_+$ for some parameter $P$. Then we will prove that if $n>2$ and $T$ is a~nondegenerate $g$-twisted trace on $\mc{A}_+$ then $g=g_t$ for some~$t$. Finally we will classify the $g_t$-twisted traces on $\mc{A}_+$ in the case when $t\neq 1$, $q^{-2}, q^{-4}, \dots$.

\begin{prop}
Let $\mc{B}$ be a filtered deformation of $A_+$. Then there exists parameter $P$ such that $\mc{B}\cong\mc{A}_+$.
\end{prop}
\begin{proof}
Choose any lifts of $Z$, $u$, $v$ in $\mc{B}$. Denote them by $Z_1$, $u_1$, $v_1$. Since $Z^iu^j$ and $Z^i v^j$ form a~basis of $A_+$, elements $Z_1^i u_1^j$ and $Z_1^i v_1^j$ form a basis of $\mc{B}$.
By definition, $u_1v_1$ equals to $q^{-n}Z_1^n+R$, where $R\in \mc{B}_{\leq 2n-2}$. Therefore $u_1v_1-q^{-n}Z_1^n$ belongs to the linear span of $Z_1^iu_1^j$ and~$Z_1^iv_1^j$, where $2i+nj<2n-2$. In particular, $j\leq 1$, so we write $u_1v_1-q^{-n}Z_1^n=R_1+u_1R_2+v_1R_3$, where $R_1,R_2,R_3\in\CN[Z_1]$ and $R_2,R_3\in\mc{B}_{\leq n-2}$. Consider $u=u_1-R_3$, $v=v_1-R_2$. We see that $uv$ belongs to $\CN[Z_1]$.
Because of commutation relations in $A_+$ we have $Z_1u-q^2uZ_1=B\in\mc{B}_{\leq n}$, $q^2Z_1v-vZ_1=C\allowbreak\in\mc{B}_{\leq n}$. In particular, $B$, $C$ belong to the span of $u, v, 1, Z_1, Z_1^2, \dots$. Let $B=b_1u+b_2v+f(Z_1)$, $C=c_1u+c_2v+g(Z_1)$.
Since $uv\in\CN[Z_1]$ we have $Z_1uv=Z_1vu$. On the other hand,
\[
Z_1uv=\big(q^2uZ_1+B\big)v=q^2uZ_1v+Bv=uvZ_1+uC+Bv.
\]
It follows that $uC+Bv=0$, so
\[
c_1u^2+c_2uv+ug(Z_1)+b_1uv+b_2v^2+f(Z_1)v=0.
\]
Since $u^2$, $v^2$, $uv$, $uZ^i$, $Z^iv$ are linearly independent in $A_+$, elements $u^2$, $v^2$, $uv$, $uZ_1^i$, $Z_1^iv$ are linearly independent in $\mc{B}$. Therefore $B=b_1u$, $C=c_2v$ and $b_1+c_2=0$. Taking $Z=Z_1+\frac{c_2}{q^2-1}$ we get $Zu=q^2uZ$, $Zv=q^{-2}vZ$.
Hence $vu$ also commutes with $Z$, so $vu\in\CN[Z]$. Let $P$ be a polynomial such that $vu=P(qz)$. From $(uv)u=u(vu)$ we get $uv=P\big(q^{-1}Z\big)$. Taking this parameter $P$ we obtain isomorphism between $\mc{A}_+$ and $\mc{B}$ that sends $u$, $v$, $Z$ to $u$, $v$, $Z$ respectively.
\end{proof}

\begin{rem}
Deformation $\mc{A}_+$ is automatically $\ZN/2\ZN$-graded with $\deg u=\deg v=n\pmod 2$, $\deg Z=0$.
\end{rem}

\begin{prop}Suppose that $n>2$ and $P(z)\neq z^n$. Let $g$ be an automorphism of $\mc{A}_+$. Then $g(Z)=\lambda Z$, $g(u)=tu$, $g(v)=\lambda^n t^{-1}v$, where $\lambda,t\in\CN^{\times}$. If there exists a $g$-twisted nondegenerate trace then $\lambda=1$ and $g=g_t$.
\end{prop}

\begin{proof}In this proof we will denote $k$-th filtration subspace of $\mc{A}_+$ by $\mc{A}_{\leq k}$.
Since $\mc{A}_{\leq 2}=\CN\cdot Z\oplus \CN\cdot 1$, we have $g(Z)=\lambda Z+b$, where $\lambda \neq 0$. For a positive integer $k$ consider the operator
\[
\phi\colon\ a\mapsto g(Z)a-q^{2k}ag(Z).
\]
Since the operator $a\mapsto Za-q^{2k}aZ$ has a kernel, namely $u^k\CN[Z]$, $\phi$ also has a kernel. We have
\begin{align*}
\phi\big(u^n R(z)\big)& =(\lambda Z+b)u^n R(Z)-q^{2k}u^n R(Z) (\lambda Z+b)
\\
& =u^n\big(\big(q^{2n}-q^{2k}\big)\lambda Z+\big(1-q^{2k}\big)b\big)R(z),
\end{align*}
where $n$ is a nonnegative integer, $R$ is any polynomial. We have a similar formula for $v^n R(z)$. It~follows that $\phi$ preserves the decomposition $\mc{A}_+=\bigoplus_{i\geq 0} u^i\CN[Z]\oplus\bigoplus_{i>0}v^i\CN[Z]$. So if $\phi$ has a~kernel then for some integer $n$ we have $\big(q^{2n}-q^{2k}\big)\lambda Z+\big(1-q^{2k}\big)b=0$. In particular, \mbox{$\big(1-q^{2k}\big)b=0$}. It~follows that $b=0$.
So $g(Z)=\lambda Z$. We have
\[
\CN u=\mc{A}_{\leq n}\cap \big\{a\in\mc{A}\colon Za=q^2aZ\big\},
\]
so $g(u)=tu$ for some $t\in \CN^{\times}$. Similarly $g(v)=sv$ for some $s\in\CN^{\times}$. Therefore
\[
P\big(q^{-1}\lambda Z\big)=g\big(P\big(q^{-1}Z\big)\big)g(uv)=g(u)g(v)=tsuv=tsP\big(q^{-1}Z\big).
\]
Suppose that the coefficient of $z^k$ in $P$ is nonzero for some $k<n$. We get $\lambda^n=ts$, $\lambda^k=ts$. So~$\lambda^{n-k}=1$, $ts=\lambda^n$.
If $\lambda=1$ we get $g=g_t$.
Suppose that $\lambda\neq 1$. Let $T$ be a $g$-twisted trace. We have $T(Za)=\lambda T(aZ)$. Since $q$ is not a root of unity, the operator $a\mapsto Za-\lambda aZ$ has image $Z\mc{A}_+$. It~follows that $Z$ belongs to the kernel of~$T$, hence~$T$ is not nondegenerate. The proposition follows.
\end{proof}

We can summarize the results in the following theorem:
\begin{thr}
Suppose that $n>2$. Nondegenerate short star-product $*$ on~$A$ such that $(A,*)$ is not isomorphic to~$A$ are classified by triples $(P,t,T)$, where $P(z)\neq z^n$ is a monic polynomial of degree $n$, $t$ is a nonzero complex number and $T$ is a $g_t$-twisted nondegenerate trace on $\mc{A}_+$.
\end{thr}

Let us classify $g_t$-twisted traces on $\mc{A}_+$ in the case when $t\neq q^{-2l}$ for all nonnegative integers~$l$. It~can be proved that in the case when $t=1$ the space of traces is infinite-dimensional, and in the case when $t=q^{-2l}$, $l>0$, the space of traces has dimension $n$, but we don't need these results.

\begin{prop}
\label{PropParametrizationOfPlusTraces}
Suppose that $t\neq q^{-2k}$ for all nonnegative integers~$k$. Then the map $T\mapsto \big(T(1),\dots,T\big(Z^{n-1}\big)\big)$ defines an isomorphism between the space of $g_t$-twisted traces and $\CN^n$.
\end{prop}
\begin{proof}
Let $T$ be a $g_t$-twisted trace. This means that $T(Za)=T(aZ)$, $T(ua)=tT(au)$, $T(va)=t^{-1}T(av)$ for all $a\in \mc{A}_+$.
The condition $T(Za)=T(aZ)$ is equivalent to $T(Z\mc{A}_{+,k})=\{0\}$ for all nonzero $k$. Here $\mc{A}_{+,k}=\mc{A}_k\cap\mc{A}_+=u^k\CN[Z]$ or $v^{-k}\CN[Z]$.
We claim that $T(\mc{A}_{+,k})=\{0\}$ for all nonzero $k$. It~is enough to prove that $T\big(u^k\big)=T\big(v^k\big)=0$ for all positive integers~$k$. Let $a=u^{k-1}$. Using $T(ua)=tT(au)$ we deduce that $(t-1)T\big(u^k\big)=0$, hence $T\big(u^k\big)=0$. We similarly prove that $T\big(v^k\big)=0$.
So we proved that $T$ is supported on $\mc{A}_{+,0}$.
Now we consider the condition $T(ua)=tT(au)$. We can assume that $a$ belongs to $\mc{A}_{+,-1}$. In~other words, $a=vR\big(q^{-1}z\big)$ for some $R\in\CN[z]$. Similarly to Proposition~\ref{PropTwistedTraceIsZeroOnPolnomials} we get
\[
T\big(P\big(q^{-1}Z\big)R\big(q^{-1}Z\big)-tP(qZ)R(qZ)\big)=0.
\]
Consider the condition $T(ua)=t^{-1}T(au)$. We can assume that $b$ belongs to $\mc{A}_{+,1}$. In~other words, $a=uR(qZ)$ for some $R\in\CN[Z]$. Similarly to Proposition~\ref{PropTwistedTraceIsZeroOnPolnomials} we get the same condition,
\[
T\big(t^{-1}P\big(q^{-1}Z\big)R\big(q^{-1}Z\big)-P(qZ)R(qZ)\big)=0.
\]
Consider a linear map $\phi\colon \CN[z]\to \CN[z]$, $\phi(S(z))=S\big(q^{-1}z\big)-tS(qz)$. We have $\phi\big(z^k\big)=q^{-k}\big(1-tq^{2k}\big)z^k$. Condition on $t$ implies that $\phi$ is a linear isomorphism. We have $P(z)\CN[z]\oplus\CN\oplus\CN z\oplus\cdots\oplus \CN z^{n-1}=\CN[z]$. Applying $\phi$ we get $\phi(P(z)\CN[z])\oplus \CN\oplus \CN z\oplus\cdots\oplus \CN z^{n-1}=\CN[z]$. Since the only condition on $T$ is that $T(\phi(P(z)\CN[z]))=0$ we deduce that $T$ is uniquely defined by $T(1), T(z),\dots, T\big(z^{n-1}\big)$.
\end{proof}

\subsection{Nondeneneracy of general trace}

\subsubsection{Construction of polynomials}
Fix $q$ such that $\abs{q}<1$. We want to prove that for Weil generic $P$, $t$, $T$ the $g_t$-twisted trace $T$ on $\mc{A}_+$ is nondegenerate.
More precisely, for $t$ not equal to $1, q^{-2}, q^{-4}, \dots$ and fixed~$P$ Proposition~\ref{PropParametrizationOfPlusTraces} says that $T\mapsto T(1),\dots, T\big(Z^{n-1}\big)$ gives an isomorphism between the space of $g_t$-twisted traces on~$\mc{A}_+$ and~$\CN^n$.

Let $P(x)=c_0+c_1x+\cdots+c_nx^n$. Then we want to prove the following result:

\begin{thr}\label{ThrWeilGenericTraceIsNondeg}
Let $\abs{q}\neq 1$. Then we can find a countable subset $Z$ of $\CN$ containing $1,q^{-2},q^{-4},\dots\!$ with the following property. For any $t\in\CN\setminus Z$ there exists a countable union of algebraic hypersurfaces $X$ in $\CN^{2n+1}$ such that for any $(c_0,\dots,c_n,t_0,\dots,t_{n-1})\in \CN^{2n+1}\setminus X$ the $g_t$-twisted trace~$T$ given by $T\big(Z^i\big)=t_i$ on the algebra $\mc{A}_+$ with parameter $P(x)=\sum_{i=0}^n c_i x^i$ is nondegenerate.
\end{thr}

Recall that $T$ is nondegenerate when for all $k\geq 0$ the bilinear form $(\cdot,\cdot)\colon \mc{A}_{+,\leq k}\times \mc{A}_{+,\leq k}\to\CN$, $(a,b)=T(ab)$ is nondegenerate. Fix $k$. Let $d$ be the dimension of $A_{+,\leq k}=A_{+,0}\oplus\cdots\oplus A_{+,k}$. Let $w_1,\dots,w_d\in\CN\langle u,v,Z\rangle$ be words in $u$, $v$, $Z$ such that their images in $A_+$ form a basis of~$A_{+,\leq k}$. It~follows that for any $P$ the images of $w_1,\dots,w_n$ in $\mc{A}_+$ form a basis of $\mc{A}_{+,\leq k}$. Form $(\cdot,\cdot)$ is nondegenerate if and only if the $d\times d$ matrix $M$, $M_{ij}=(w_i,w_j)$ has nonzero determinant for all $k\geq 0$.

We claim the following:
\begin{lem}
\label{LemDeterminantIsPolynomial}
Let $1\leq i,j\leq d$. Then the matrix coefficient $M_{ij}=T(w_iw_j)$ is a polynomial in $c_0,\dots,c_n, T(1), T(Z),\dots, T\big(Z^{n-1}\big), t, \frac{1}{1-t},\frac{1}{1-q^2t}, \dots, \frac{1}{1-q^{4k}t}$. As a corollary, the determinant of~$M$ is also a polynomial in these variables.
\end{lem}

\begin{proof}Fix $i$, $j$. We make $\CN\langle u,v,Z\rangle$ into a $\ZN$-graded algebra by $\deg u=1$, $\deg v=-1$, \mbox{$\deg Z=0$}. Since $w_1,\dots,w_d$ are words, they are homogeneous. In the case when $\deg w_i+\deg w_j\neq 0$ we have $T(w_iw_j)=0$.

Consider the case when $\deg w_i+\deg w_j=0$. Applying relations $uv=P(q^{-1}Z)$, $vu=P(qZ)$, $Zu=q^2uZ$, $Zv=q^{-2}vZ$ we obtain $w_sw_t=A_{i,j}(Z,c_0,\dots,c_n)$, where $A_{i,j}$ is a fixed polynomial.
We write $A_{i,j}(Z,c_0,\dots,c_n)=\sum_{l=0}^{2k} A_{i,j,l}(c_0,\dots,c_n)Z^l$.

Consider the map $\phi$ from the proof of Proposition~\ref{PropParametrizationOfPlusTraces}.
It follows from the proof of Proposition~\ref{PropParametrizationOfPlusTraces} that we can compute $T(w_iw_j)$ as follows: take $\phi^{-1}(A_{i,j})$, take remainder from division by $P$, apply $\phi$, apply $T$. We see that \(\phi^{-1}(A_{i,j,l})Z^l=\frac{1}{1-q^{2l}t}A_{i,j,l}Z^l.\) Hence $\phi^{-1}(A_{i,j})$ is a polynomial in $Z,c_0,\dots,c_n,t, \frac{1}{1-t}, \frac{1}{1-q^2t}, \dots, \frac{1}{1-q^{4k}t}$. Doing long division by $P$ we obtain a polynomial $B_{i,j}$ of degree less than $n$. Applying $\phi$ and then~$T$ gives us a polynomial in $c_0,\dots,c_n, T(1), T(Z),\dots, T\big(Z^{n-1}\big), t, \frac{1}{1-t}, \frac{1}{1-q^2t}, \dots, \frac{1}{1-q^{4k}t}$ as claimed.
\end{proof}

Lemma~\ref{LemDeterminantIsPolynomial} says that we can choose a polynomial with coefficients in $\CN(t)$ $Q_k\in \CN(t)[c_0,\dots,c_n,\allowbreak T_0,\dots,T_{n-1}]$ such that{\samepage
\[
\det M=Q_k(t)\big(c_0,\dots,c_n,T(1),\dots,T\big(Z^{n-1}\big)\big).
\]
Moreover, when $t\neq q^{-2l}$ for all $l\geq 0$, $Q_k(t)$ is a well-defined polynomial.}

Theorem~\ref{ThrWeilGenericTraceIsNondeg} is a corollary of the following proposition:
\begin{prop}
\label{PropQkIsNonzero}
$Q_k$ is a nonzero element of $\CN(t)[c_0,\dots,c_n,t_0,\dots,t_{n-1}]$.
\end{prop}

Indeed, if Proposition~\ref{PropQkIsNonzero} is true then for each $Q_k$ there exist a finite set $Z_k\subset \RN$ such that~$Q_k(t)$ is a nonzero polynomial when $t\in \RN\setminus Z_k$. In this case we can take $Z=\big\{1,q^{-2},q^{-4},{\dots}\big\}\allowbreak\cup\bigcup_{k=0}^{\infty}Z_k$. Then for each $t\in\RN\setminus Z$ all $Q_k(t)$ are nonzero. We can take $X$ is a union of hyper\-surfaces defined by $Q_k(t)$ for all $k\geq 0$.

Let $(c_0,\dots,c_n,t_0,\dots,t_{n-1})\in\CN^{2n+1}\setminus X$. Suppose that $\mc{A}_{+}$ is the algebra corresponding to $P=\sum c_i x^i$, $T$ is a $g_t$-twisted trace defined by $T\big(Z^i\big)=t_i$. By definition of $X$ for all $k\geq 0$ we have $\det M_k=Q_k(t,c_0,\dots,c_n,t_0,\dots,t_{n-1})\neq 0$. This means that the form $(\cdot,\cdot)\colon\mc{A}_{+,\leq k}\times\mc{A}_{+,\leq k}\to\CN$ given by $(a,b)=T(ab)$ is nondegenerate. Hence the trace $T$ is nondegenerate, as required.

In order to prove Proposition~\ref{PropQkIsNonzero} it is enough to take $t=q^2$. To prove that $Q_k\big(q^2\big)$ is not identically zero we will prove that there exists a polynomial $P=\sum c_i x^i$ and a $g_{q^2}$-twisted non\-de\-ge\-ne\-rate trace $T$ on $\mc{A}_+$. This means that for all $k\geq 0$ we have $Q_k\big(c_0,\dots,c_n,T(1),\dots,T\big(Z^{n-1}\big)\big)=\det M_k\neq 0$, hence all $Q_k$ are not identically zero.

We find the required $P$, $t$ in Theorem~\ref{ThrGenericQ2TraceIsNondeg} below.

\subsubsection{Construction of one nondegenerate trace}
 Let $a$ be a complex number such that $\abs{q}^2<\abs{a}<1$, $b=q^2a^{-1}$. We will take $P$ so that $P$ is divisible by $(z-a)(z-b)$.
Let $w$ be a meromorphic function on $\CN\setminus \{0\}$ such that $w\big(q^2x\big)=w(x)$ and all poles of $w$ are simple poles at $q^{2k}a$ and $q^{2k}b$ for all integer $k$.
Consider
\[
T(R)=\int_{S^1}R(z) w(z)\, {\rm d}z=2\pi\rmi\int_0^1 R\big({\rm e}^{2\pi\rmi x}\big)w\big({\rm e}^{2\pi\rmi x}\big){\rm e}^{2\pi\rmi x}{\rm d}x.
\]
Using Theorem~\ref{ThrTwistedTraceAndQuasiPeriodic} we deduce that $T$ is a $g_t$-twisted trace on $\mc{A}$.

Denote by $A$, $B$ the residues of $w$ in $a$, $b$. We see that $u=w(e^z)$ is an elliptic function in~$z$. Hence the sum of residues of $w_1$ at $\ln a$, $\ln b$ equals to zero.
We write $w(x)=\frac{w_1(x)}{x-a}$, where $w_1(a)=A$. Hence $u(z)=\frac{w_1(e^z)}{e^z-a}$. We see that the residue of $u(z)$ at $\ln a$ equals to $\frac{A}{a}$. We similarly deduce that the residue of $u(z)$ at $\ln b$ equals to $\frac{B}{b}$. It~follows that $\frac{A}{a}+\frac{B}{b}=0$.

Since $w\big(q^2x\big)=w(x)$, the residue in $q^{2k}a$ equals to $q^{2k} A$ and the residue in $q^{2k}b$ equals to~$q^{2k} B$. This allows us to write the Mittag-Leffler series for $w$:

\begin{prop}$w(x)=\sum_{k\in\ZN} \big(\frac{q^{2k}A}{x-q^{2k}a}+\frac{q^{2k}B}{x-q^{2k}b}\big)+C$, where $C\in\CN$.
\end{prop}

\begin{proof}
The right-hand side is convergent for $k\geq 0$ and for $k<0$ we use $\frac{A}{a}+\frac{B}{b}$ to write \[
\frac{q^{2k}A}{x-q^{2k}a}+\frac{q^{2k}B}{x-q^{2k}b}=\frac{Ax}{a\big(x-q^{2k}a\big)}+\frac{Bx}{b\big(x-q^{2k}b\big)}.
\]
It follows that both sides are elliptic with the same poles. The proposition follows.
\end{proof}

We want to write the Laurent expansion for $w$ on $S^1$. For $k<0$ we have $\big|aq^{2k}\big|,\big|bq^{2k}\big|>1$, so
\[
\frac{1}{x-q^{2k}a}=\sum_{i\geq 0} x^{i} \big(q^{2k}a\big)^{-i-1}.
\]
For $k\geq 0$ we have $\big|aq^{2k}\big|,\big|bq^{2k}\big|<1$, so
\[
\frac{1}{x-q^{2k}a}=\sum_{i\leq 0} x^{-i-1} \big(q^{2k}a\big)^i.
\]
Therefore the coefficient of the Laurent expansion of $w$ on $x^{-i-1}$ equals to
\[
\sum_{k\geq 0} q^{2k}A \big(q^{2k}a\big)^i+\sum_{k\geq 0} q^{2k} B\big(q^{2k}b\big)^i=\frac{Aa^i+Bb^i}{1-q^{2(i+1)}}.
\]

Multiplying $w$ by a constant if necessary we can assume that $\frac{A}{a}=-\frac{B}{b}=1$. This gives $Aa^i=a^{i+1}$, $Bb^i=-b^{i+1}$. Moreover, $ab=q^2$. So we have proved the following lemma:
\begin{lem}
The coefficient of $z^{-i-1}$ in the Laurent expansion of $w$ on $S^1$ equals to
\[
\frac{a^{i+1}-q^{2i+2}a^{-i-1}}{1-q^{2(i+1)}}.
\]
In particular, it is a Laurent polynomial of degree $i+1$ in $a$ with leading term equal to $\frac{1}{1-q^{2(i+1)}}$.
\end{lem}

\begin{prop}\label{PropHilbertMatrix}
Let $R$ be a nonzero polynomial, $m$ be a nonnegative integer, $M$ be the matrix with entries $M_{ij}=\int_{S^1} R(z)w(z)z^{i+j}{\rm d}z$, $0\leq i,j\leq m$. Then $\det M$ is a nonzero Laurent polynomial in $a$.
\end{prop}

\begin{rem}
This matrix $M$ does not coincide with the matrix $M$ from the previous section.
\end{rem}

\begin{proof}
We may assume that $R$ is a monic polynomial of degree $d$. Then the coefficient of~$R(z)w(z)$ on $z^{-i}$ is a Laurent polynomial in $a$ with degree $d+i$ and leading term equal to~$\frac{1}{1-q^{2i+2d}}$.
We see that $M_{ij}$ equals to the coefficient of $z^{-i-j-1}$ in the Laurent expansion of $R(z)w(z)$ on~$S^1$. It~follows that $M_{ij}$ is a Laurent polynomial in $a$ with degree $d+i+j+1$ and leading term $\frac{1}{1-q^{2i+2j+2d+2}}$. We see that the coefficient of $a^{(m+1)d+m^2+m}$ in $\det M$ equals to $\det M_0$, where $M_0$ is a matrix with $(M_0)_{ij}=\frac{1}{1-q^{2i+2j+2d+2}}$.
Multiplying $M_0$ by a diagonal matrix~$D$ with $D_{ii}=q^{2i}$ we get $(M_0D)_{ij}=\frac{1}{q^{-2j}-q^{2i+2d+2}}$. Taking $x_i=q^{2i+2d+2}$, $y_i=q^{-2j}$ we get that $(M_0D)_{ij}=\frac{-1}{x_i-y_j}$. Hence up to a sign and a power of $q$ the determinant of $M_0$ equals to the determinant of a Cauchy matrix $\big(\frac{1}{x_i-y_j}\big)_{i,j=1}^n$. The determinant of a Cauchy matrix is equal to $\prod_{i\neq j}(x_i-x_j)(y_i-y_j)\prod_{i,j}(x_i-y_j)^{-1}$. In particular, $\det M_0$ is not equal to zero.
\end{proof}

Now we are ready to prove one of the main theorems:
\begin{thr}
\label{ThrGenericQ2TraceIsNondeg}
Suppose that $|q|\neq 1$, $t=q^2$, $P$, $T$ are as above. Then there exists $a$ such that~$T$ is nondegenerate.
\end{thr}

\begin{proof}
Denote the space of $b\in \mc{A}_+$ such that $KbK^{-1}=q^{2i}b$ by $\mc{A}_i$. If $b_1\in\mc{A}_i$, $b_2\in\mc{A}_j$ then $(b_1,b_2)=0$ unless $i+j=0$. So we may assume that $i\geq 0$, $i+j=0$ and prove that $(\cdot,\cdot)$ gives a nondegenerate pairing between $\mc{A}_{i,\leq m}$ and $\mc{A}_{j,\leq m}$. We have
\begin{gather*}
\mc{A}_{i,\leq m}=\Span \big(u^i, u^iZ, \cdots, u^i Z^{\lfloor \frac12 m-in\rfloor}\big),
\\
\mc{A}_{j,\leq m}=\Span\big(v^i, v^iZ, \cdots, v^i Z^{\lfloor \frac12 m-in\rfloor}\big).
\end{gather*}
Let $d$ be a positive integer. Consider a form $B\colon\CN[x]_{\leq d}\times\CN[x]_{\leq d}\to \CN$,
\[
B(R(x),S(x))=\big(R(Z)u^i,S(Z)v^i\big).
\] It is enough to prove that $B$ is nondegenerate for any $d$ and $i$. We have
\[
B\big(R(x),S(x)\big)=T\big(R(Z)u^i S(Z)v^i\big)=
T\big(R(Z)S\big(q^{-2i}Z\big)u^iv^i\big)=T\big(R(z)S\big(q^{-2i}z\big)Q_0(z)\big),
\]
where $Q_0$ is some fixed nonzero polynomial. Since $T(R(z))=\int_{S^1} R(z)w(z){\rm d}z$, it remains to prove that integration with weight $Q_0(z)w(z)$ gives a nondegenerate bilinear form on $\CN[z]_{\leq d}$. It~follows from Proposition~\ref{PropHilbertMatrix} that this form is nondegenerate for all but a finite number of~$a$. There is countable number of pairs~$(i,d)$, this gives countable number of~$a$ that don't work. Hence there exists $a$ such that $T$ is nondegenerate.
\end{proof}

\section[The case of central reduction of U\textunderscore{}q(sl\textunderscore{}2)]
{The case of central reduction of $\boldsymbol{U_q(\mf{sl}_2)}$}\label{SecSl2}

\subsection{Unitarizability in terms of the value of Casimir element}
Suppose that $0<q<1$. The algebra $U_{q}(\mf{sl}_2)$ is generated by $E$, $F$, $K$ with relations $KEK^{-1}\allowbreak=q^2E$, $KFK^{-1}=q^{-2}F$, $[E,F]=\frac{K-K^{-1}}{q-q^{-1}}$.
The center of $U_{q}(\mf{sl}_2)$ is generated by
\[
\Omega=FE+\frac{Kq+K^{-1}q^{-1}-2}{\big(q-q^{-1}\big)^2}.
\]
Consider $\mc{A}=U_{q}(\mf{sl}_2)/(\Omega-c)$, where $c\in \RN$. We see that $\mc{A}$ is a $q$-deformation with parameter
\[
P(qx)=-\frac{xq+x^{-1}q^{-1}-2}{\big(q-q^{-1}\big)^2}+c,
\]
so
\[
P(x)=-\frac{x+x^{-1}-2}{\big(q-q^{-1}\big)^2}+c.
\]
It follows that $P(x)=\ovl{P}(x)=\ovl{P}\big(x^{-1}\big)$. We see that $\mc{A}$ has conjugation $\rho$ such that $\rho(E)=F$, $\rho(F)=E$, $\rho(K)=K^{-1}$.
Suppose that $\mc{A}$ is a generalized $q$-Weyl algebra with parame\-ter~$P(x)$. When $P(x)=\ovl{P}\big(x^{-1}\big)$, the algebra $\mc{A}$ has a conjugation $\rho$ such that $\rho(u)=v$, $\rho(v)=u$, $\rho(Z)=Z^{-1}$.
We also assume that the total degree $n$ of~$P$ equals to $2$. In other words, $P$~belongs to the span of $1$, $x$, $x^{-1}$.

\begin{lem}Let $\mc{A}$ be a $q$-deformation with $n=2$ and parameter $P$ such that $P(x)=\ovl{P}\big(x^{-1}\big)$. Then there exists $c\in \RN$ such that
\[
\mc{A}\cong U_{q}(\mf{sl}_2)/(C-c).
\]
Moreover, this isomorphism intertwines the corresponding conjugations.
\end{lem}
\begin{proof}
Suppose that $s$ is a complex number such that $|s|=1$. Since $\ovl{Z}=Z^{-1}$, we have $\ovl{s Z}=(s Z)^{-1}$. So if we change $P(x)$ to $P(sx)$, the algebra $\mc{A}$ and the conjugation $\ovl{\cdot}$ will be the same.
Suppose that $a\in \RN^{\times}$. Since $\ovl{au}=av$, $\ovl{av}=au$, we can change $P(x)$ to $a^2P(x)$ and the conjugation will be the same. So the polynomials $P(x)$ and $a^2P(sx)$ give the same $q$-deformation and the same conjugation for all $a\in \RN^{\times}$, $s\in \CN$ such that $|s|=1$.
We want to find $s$, $a$, $c$ such that
\[
\pm a^2P(sx)=-\frac{x+x^{-1}-2}{\big(q-q^{-1}\big)^2}+c.
\]
If we find such $s$, $a$, $c$, we change $P(x)$ to $a^2 P(sx)$ and define
\[
\phi\colon\ \mc{A}\to U_q(\mf{sl}_2)/(C-c)
\]
by $\phi(u)=E$, $\phi(v)=F$, $\phi(Z)=K$.
We choose $a\in\RN$ and $s=\pm 1$ so that the leading coefficient of $a^2P(sx)$ equals to $\frac{1}{(q-q^{-1})^2}$. Then we change $P(x)$ to $a^2P(sx)$. Define $c$ to be the coefficient of $P$ on $1$ minus $\frac{2}{(q-q^{-1})^2}$. It~follows from $P(x)=\ovl{P}\big(x^{-1}\big)$ that
\[
P(x)=-\frac{x+x^{-1}-2}{\big(q-q^{-1}\big)^2}+c.
\]
Therefore
\[\mc{A}\cong U_{q}(\mf{sl}_2)/(C-c)
\]
and this isomorphism respects conjugation.
\end{proof}

Recall that $\mc{A}$ has a positive definite invariant form if and only if both roots $\alpha_1$, $\alpha_2$ of $P$ satisfy $q<\abs{\alpha_i}<q^{-1}$, $i=1,2$.
By Vieta's formulas we have $\alpha_1\alpha_2=1$, $\alpha_1+\alpha_2=c\big(q-q^{-1}\big)^2+2$. In particular, $q<\abs{\alpha_i}<q^{-1}$ if and only if $\big|c\big(q-q^{-1}\big)^2+2\big|<q+q^{-1}$. This is equivalent to $c\big(q-q^{-1}\big)^2+2\in \big({-}q-q^{-1},q+q^{-1}\big)$. We get $c\in \big(\frac{-q-q^{-1}-2}{(q-q^{-1})^2}, \frac{q+q^{-1}-2}{(q-q^{-1})^2}\big)$.
We also have $\abs{\alpha_1}=\abs{\alpha_2}=1$ if and only if $\big|c\big(q-q^{-1}\big)^2+2\big|\leq 2$. So when $c\in \big({-}\frac{4}{(q-q^{-1})^2},0\big)$ both roots lie on the circle.

\subsection{Computation of the cone of positive functions in a fixed basis}

We assume that $c\in \big(\frac{-q-q^{-1}-2}{(q-q^{-1})^2}, \frac{q+q^{-1}-2}{(q-q^{-1})^2}\big)$. Denote the roots of $P(x)$ by $z_1$, $z_2$. Since $P(x)=\ovl{P}(x)=P\big(x^{-1}\big)$ we have $z_1z_2=1$ and either $z_1=\ovl{z_2}$ or $z_1,z_2\in \RN$.
The linear space of traces on $\mc{A}$ is isomorphic to the space $L$ of elliptic functions $w$ with period $q^2$ such that $P(z)w(qz)$ is holomorphic on the set $\big\{z\mid q\leq \abs{z}\leq q^{-1}\big\}$. This space has dimension $2$. Function $w$ corresponds to the trace $T(R)=\int_{S^1}R(z)w(z)\abs{{\rm d}z}$.

Let $\wp$ be a Weierstrass elliptic function with periods $2\ln q$, $2\pi {\rm i}$ and $w_0(z)=\wp(\ln z)$. Since $z_1z_2=1$, we have $w_0(qz_1)=w_0(qz_2)$. Since $z_1$ and $z_2$ belong to $S^1$ or $\RN$, we also have $c_0=w_0(qz_1)\in \RN$. Then $L$ has a basis $1$, $\frac{1}{w_0-c_0}$.
Consider a real linear subspace $L_{\RN}$ consisting of functions $w$ real on $S^1$. It~can be written as
\[
L_{\RN}=\bigg\{a+\frac{b}{w_0-c_0}\,\bigg|\, a,b\in \RN\bigg\}.
\]
We have $2\ln q\in\RN$, $2\pi{\rm i}\in {\rm i}\RN$. In this case it follows from the definition that $\wp(x)$ is real when $x\in \ln q\ZN+{\rm i}\RN$ or $x\in \pi{\rm i}\ZN+{\rm i}\RN$. This means that rays $I_1=w((0,2\ln q])$, $I_2=w((0,2\pi \rmi])$ and intervals $I_3=w([\pi\rmi,\pi\rmi+2\ln q])$, $I_4=w([\ln q,\ln q+2\pi\rmi])$ belong to the real line.

The derivative $\wp'$ of $\wp$ has roots at $\pi\rmi$, $\ln q$ and $\pi\rmi+\ln q$. It~follows that $\wp$ takes each value in the interior of $I_1$, $I_2$, $I_3$, $I_4$ two times. Since $\wp$ takes each value at most two times we deduce that interiors of $I_1$, $I_2$, $I_3$, $I_4$ do not intersect.
It follows from the definition of $\wp$ that $\wp$ is positive on $(0,\eps)$ and negative on $(0,\rmi\eps)$ for small enough $\eps>0$. Hence $I_2=(-\infty,e_1]$ and $I_1=[e_3,\infty)$. Since~$I_3$ intersects with~$I_2$ only in the end point we can write $I_3=[e_3,e_2]$, where $e_3<e_2$. Since~$I_4$ interects with $I_1$, $I_3$ only in the end points we can write $I_2=[e_2,e_1]$, so that $e_2<e_1$.
We deduce that
\begin{gather*}
w_0\big(qS^1\big)=\wp([\ln q,\ln q+2\pi\rmi])=[e_2,e_3],
\\
w_0\big(S^1\big)=\wp([0,2\pi\rmi])=[-\infty, e_1],
\\
w_0(\RN_{>0})=\wp([0,2\ln q])=[e_3,\infty],
\\
w_0(\RN_{<0})=\wp([\pi\rmi,\pi\rmi+2\ln q])=[e_1,e_2].
\end{gather*}
There exists a transcendental formula that expresses $e_1$, $e_2$, $e_3$ through $2\pi {\rm i}$ and $2\ln q$.
Theorem~\ref{ThrPositivityThroughQuasiPeriodicFunctions} says that elliptic function $w$ gives a positive definite trace when $w(z)$ and $P(z)w(qz)$ are nonnegative on~$S^1$.
There are six cases depending on the leading sign of $P$ and $c_0$: $c_0\in [e_3,\infty)$, $c_0\in [e_2,e_3]$, $c_0\in (e_1,e_2]$. This corresponds to roots on $\RN_{>0}$, $S^1$, $\RN_{<0}$ respectively.

Let $w_1=\frac{1}{w_0-c_0}$.
Consider the case when $c_0\in (e_2,e_3)$. In this case the roots $z_1$, $z_2$ divide $S^1$ into two contours~$C_+$ and $C_-$. Then $w_1\big(S^1\big)=\big[\frac{1}{e_1-c_0},0\big)$, $w_1(qC_+)=\big[\frac{1}{e_3-c_0},\infty\big)$, $w_1(qC_-)=\big({-}\infty,\frac{1}{e_2-c_0}\big)$.
The positivity condition for the elliptic function $w$ is as follows: $w$ is nonnegative on $S^1$, $w(qz)$ is nonnegative on $C_{\pm}$ and nonpositive on $C_{\mp}$, or vice versa, depending on the leading sign of $P$.
The space $L_{\RN}$ consists of functions $a+bw_1$ for $a,b\in \RN$. In the case when $w$ is nonnegative on $C_+$ the answer is
\[
\bigg\{\lambda(a+w_1)\bigg| \lambda>0, a\in \bigg[\frac{1}{c_0-e_1},\frac{1}{c_0-e_2}\bigg]\bigg\}.
\]
In the case when $w$ is nonpositive on $C_+$ the answer is
\[
\bigg\{{-}\lambda(a+w_1)\bigg|\lambda>0, a\in \bigg[{-}\frac{1}{e_3-c_0},0\bigg]\bigg\}.
\]
The other cases of $c_0$ can be done similarly.

\subsection[Nondegeneracy of U\textunderscore{}q(sl\textunderscore{}2)-invariant traces]
{Nondegeneracy of $\boldsymbol{U_q(\mf{sl}_2)}$-invariant traces}

Let $\mc{A}$ be a central reduction of $U_q(\mf{sl}_2)$:
\[
\mc{A}=U_q(\mf{sl}_2)/(C-c_0),
\]
where $C=FE+\frac{Kq+K^{-1}q^{-1}}{(q-q^{-1})^2}$. The adjoint action of $U_q(\mf{sl}_2)$ on itself gives adjoint action of~$U_q(\mf{sl}_2)$ on $\mc{A}$. The subalgebra $\mc{A}_+$ of locally finite elements is generated by $F$, $K^{-1}$ and $EK^{-1}$. Consider a trace $T_0$ on $\mc{A}$ and a $g_{q^{-2}}$-twisted trace $T(a)=T_0\big(K^{-1}a\big)=T_0\big(aK^{-1}\big)$. We note that~$T$ is invariant with respect to the adjoint action:
\begin{gather*}
T\big(KaK^{-1}\big)=T(a)=\eps(K)T(a),
\\
T\big(Ea-KaK^{-1}E\big)=T_0\big(K^{-1}Ea-aK^{-1}E\big)=0,
\\
T((Fa-aF)K)=T_0(Fa-aF)=0.
\end{gather*}

Since $\mc{A}_+$ is the sum of all finite-dimensional irreducible representations of $U_q(\mf{sl}_2)$ of type $1$ and odd dimension, we deduce that $T|_{\mc{A}_+}$ is uniquely defined up to a constant. This is easy to see directly: a space of $g_{q^{-2}}$-twisted traces on $\mc{A}$ is two-dimensional, but one of the traces sends $\sum a_i K^i$ to $a_1$, so the restriction of this trace on $\mc{A}_+$ is zero.

We turn to the question when $T$ gives a nondegenerate short star-product.
\begin{prop}
Consider the bilinear form $(a,b)=T(ab)$. For $c_0$ outside of a countable subset of $\CN$ the restriction of this form to $\mc{A}_{\leq i}$ is nondegenerate for all $i$. More precisely, the following three statements are equivalent:
\begin{enumerate}\itemsep=0pt
\item[$1.$]
The restriction of $(\cdot,\cdot)$ to $\mc{A}_{\leq i}$ is nondegenerate for all $i$.
\item[$2.$]
$(\cdot,\cdot)$ has trivial kernel.
\item[$3.$]
$\mc{A}$ does not have finite-dimensional representations.
\end{enumerate}
In particular, for generic $c_0$ the unique $U_q(\mf{sl}_2)$-invariant trace gives a nondegenerate short star-product on $A=\gr\mc{A}_+$.
\end{prop}
\begin{proof}
There are two representation of $U_q(\mf{sl}_2)$ of a given dimension $d$. Hence the set of $c_0$ such that $\mc{A}$ has a finite-dimensional representation is countable. So it is enough to prove equivalence of the three statements.

We start with $(1)\Leftrightarrow (2)$. It~is enough to prove that $(1)$ follows from $(2)$. Suppose that this is not the case and $(\cdot,\cdot)|_{\mc{A}_{\leq i}}$ is degenerate for some $i$.
The multiplication map is $\ad$-equivariant, so the map $\phi\colon a\otimes b\mapsto T(ab)$ is an $U_q(\mf{sl}_2)$-equivariant map from $\mc{A}_+\otimes\mc{A}_+$ to $\CN$. Let $\mc{A}_+=\bigoplus_{k\geq 0}V_{2k}$, where $V_l$ is a type $1$ irreducible representation of $U_q(\mf{sl}_2)$ of dimension $l+1$. It~follows from Schur lemma that the restriction of~$\phi$ to $V_k\otimes V_l$ is either zero or nondegenerate bilinear map. In particular, for $k\neq l$ it is zero.
It follows that $(\cdot,\cdot)$ is denegerate on some $V_i$ if and only if $T(ab)=0$ for all $a\in V_i$, $b\in\mc{A}_+$. Note that $\mc{A}_{\leq i}=V_0\oplus\cdots\oplus V_i$. So if $(\cdot,\cdot)$ restricted to $\mc{A}_{\leq i}$ is degenerate for some $i$ it follows that $(\cdot,\cdot)$ has a nontrivial kernel.

Now we prove $(2)\Leftrightarrow (3)$. We start with $(2)\Rightarrow (3)$.
Suppose that $(\cdot,\cdot)$ has a kernel $I$. Since $(a,b)=T(ab)$, for any $\ad K$-homogeneous $c\in\mc{A}$ we have
\begin{gather*}
(ac,b)=T(acb)=(a,cb),
\qquad
(ca,b)=T(cab)=\lambda T(ab c)=(a,bc)
\end{gather*}
for some $\lambda\in\CN$. In particular, $I$ is an ideal in $\mc{A}_+$. From the proof of $(1)\Leftrightarrow (2)$ we see that~$I$ has a finite codimension in $\mc{A}_+$. Also~$I$ contains $V_k$ for some~$k$, hence it contains some powers of~$E$ and~$FK^{-1}$.
Let $V$ be a finite-dimensional irreducible representation of $\mc{A}_+$ in the composition series of~$\mc{A}_+/I$. We see that $E$ and $FK^{-1}$ act nilpotently on~$V$.
Let $V_0$ be the kernel of $K^{-1}$. It~is easy to see that $V_0$ is a subrepresentation of $\mc{A}_+$. So~$V_0=V$ or $V_0=\{0\}$. Since $\mc{A}$ is $\mc{A}_+$ localized in $K^{-1}$ we see that in the second case $V_0$ is a~finite-dimensional representation of $\mc{A}$.
Now we prove that the first case is impossible. Let $R$ be a polynomial such that $EFK^{-1}=R\big(K^{-1}\big)$. We have $R(x)=\frac{x^2q+q^{-1}}{(q-q^{-1})^2}+c_0 x$. But $FK^{-1}$ is nilpotent operator, so $EFK^{-1}$ has a~nontrivial kernel on $V_0$, while $R\big(K^{-1}\big)=\frac{q^{-1}}{(q-q^{-1})^2}$ does not. We get a contradiction.

It remains to prove $(3)\Rightarrow (2)$. Suppose that $\mc{A}$ has a finite-dimensional representation $V$. Denote by $I$ the intersection of its kernel with $\mc{A}_+$. This is a two-sided ideal in $\mc{A}_+$.
The quantum trace on $V$ is an $U_q(\mf{sl}_2)$-invariant map from $\mc{A}$ to $\CN$ that is nonzero on~$\CN\big[K^{-1}\big]$. Therefore its restriction to $\mc{A}_+$ coincides with $T$ up to a multiplication by a scalar. It~follows that $I$ belongs to the kernel of~$T$. The proposition follows.
\end{proof}

\subsection*{Acknowledgments}
I am grateful to Pavel Etingof for formulation of the problem, stimulating discussions and for helpful remarks on the previous versions of this paper. I would like to thank Mykola Dedushenko for explaining the physical meaning of positive traces on generalized $q$-Weyl algebras. I am grateful to the anonymous reviewers for helpful remarks and suggestions.

\pdfbookmark[1]{References}{ref}
\LastPageEnding

\end{document}